\documentclass[reqno]{amsart}

\usepackage[numbers]{natbib}
\usepackage{amssymb,mathtools}
\usepackage[shortlabels]{enumitem}
\usepackage{float}
\usepackage{caption}
\usepackage{footnote}
\usepackage{tikz, tikz-cd}
\usetikzlibrary{matrix, arrows}
\usetikzlibrary{patterns}
\usepackage{wrapfig}
\usepackage{hyperref}
\usepackage{array}
\hypersetup{colorlinks,linkcolor={blue},citecolor={blue},urlcolor={blue}}
\usepackage[cal=boondox]{mathalfa}
\usepackage{graphicx}
\usepackage[abs]{overpic}
\graphicspath{ {images/} }

\newtheorem{theorem}{Theorem}[section]

\newtheorem{lemma}[theorem]{Lemma}
\newtheorem{proposition}[theorem]{Proposition}
\newtheorem{corollary}[theorem]{Corollary}

\theoremstyle{definition}
\newtheorem{definition}[theorem]{Definition}
\newtheorem{notation}[theorem]{Notation}

\numberwithin{equation}{section}

\newcommand{\bR}{\mathbb{R}}
\newcommand{\bQ}{\mathbb{Q}}
\newcommand{\bZ}{\mathbb{Z}}
\newcommand{\bN}{\mathbb{N}}
\newcommand{\bC}{\mathbb{C}}

\newcommand{\cO}{\mathcal{O}}
\newcommand{\cP}{\mathcal{P}}

\newcommand{\cS}{\mathcal{S}}
\newcommand{\cA}{\mathcal{A}}
\newcommand{\fp}{\mathfrak{p}}

\newcommand{\fa}{\mathfrak{a}}
\newcommand{\fd}{\mathfrak{d}}

\newcommand{\fD}{\mathfrak{D}}

\newcommand{\bv}{\mathbf{v}}

\newcommand{\eps}{\varepsilon}

\newcommand{\tr}{\text{tr}}

\begin{document}

\title[Geometry]{A geometric study of circle packings and ideal class groups}

\thanks{The author is grateful to Elena Fuchs and Katherine Stange for their guidance and support, to Katherine Stange and Xin Zhang for verifying the switch to $\text{PGL}$, to Alex Kontorovich for helping with connections to Kleinian packings/bugs, and to Robert Hines for many helpful conversations.}

\author{Daniel E. Martin}
\address{University of California, Davis, CA, United States}
\email{dmartin@math.ucdavis.edu}

\subjclass[2010]{Primary: 52C26, 11R11, 30F40, 11R29. Secondary: 20G30, 20F65, 11E39.}

\keywords{imaginary quadratic field, circle packing, Bianchi group, class group}

\date{\today}

\begin{abstract}A family of fractal arrangements of circles is introduced for each imaginary quadratic field $K$. Collectively, these arrangements contain (up to an affine transformation) every set of circles in $\widehat{\bC}$ with integral curvatures and Zariski dense symmetry group. When that set is a circle packing, we show how the ambient structure of our arrangement gives a geometric criterion for satisfying the almost local-global principle. Connections to the class group of $K$ are also explored. Among them is a geometric property that guarantees certain ideal classes are group generators.\end{abstract}

\maketitle

\section{Introduction}\label{sec:1}

Let $K$ be an imaginary quadratic field with ring of integers $\cO$ and discriminant $\Delta$. To each $K$ we associate a family of arrangements, meaning a set of oriented circles in $\widehat{\bC}$, the extended complex plane. We denote a member of this family $\cS_D$, where $D$ is an integer whose allowable values depend on $K$. Figure \ref{fig:1} shows $\cS_{1}$ and $\cS_{24}$ for $\bQ(i\sqrt{19})$, to be discussed in Subsection \ref{ss:dich}.

\begin{figure}[ht]
    \vspace{-0.2cm}
    \centering
    \includegraphics[width=0.48\textwidth]{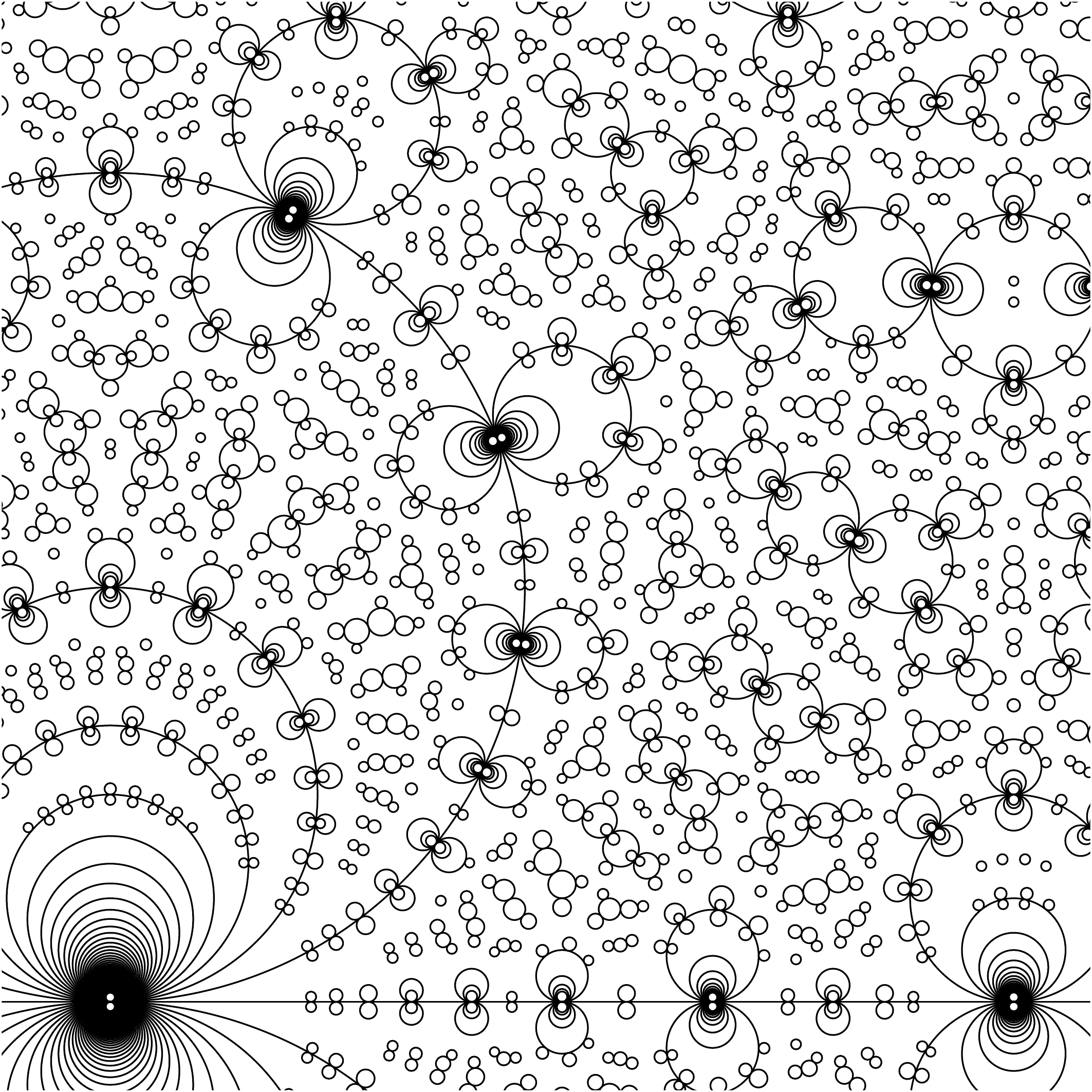}\hspace{0.04\textwidth}\includegraphics[width=0.48\textwidth]{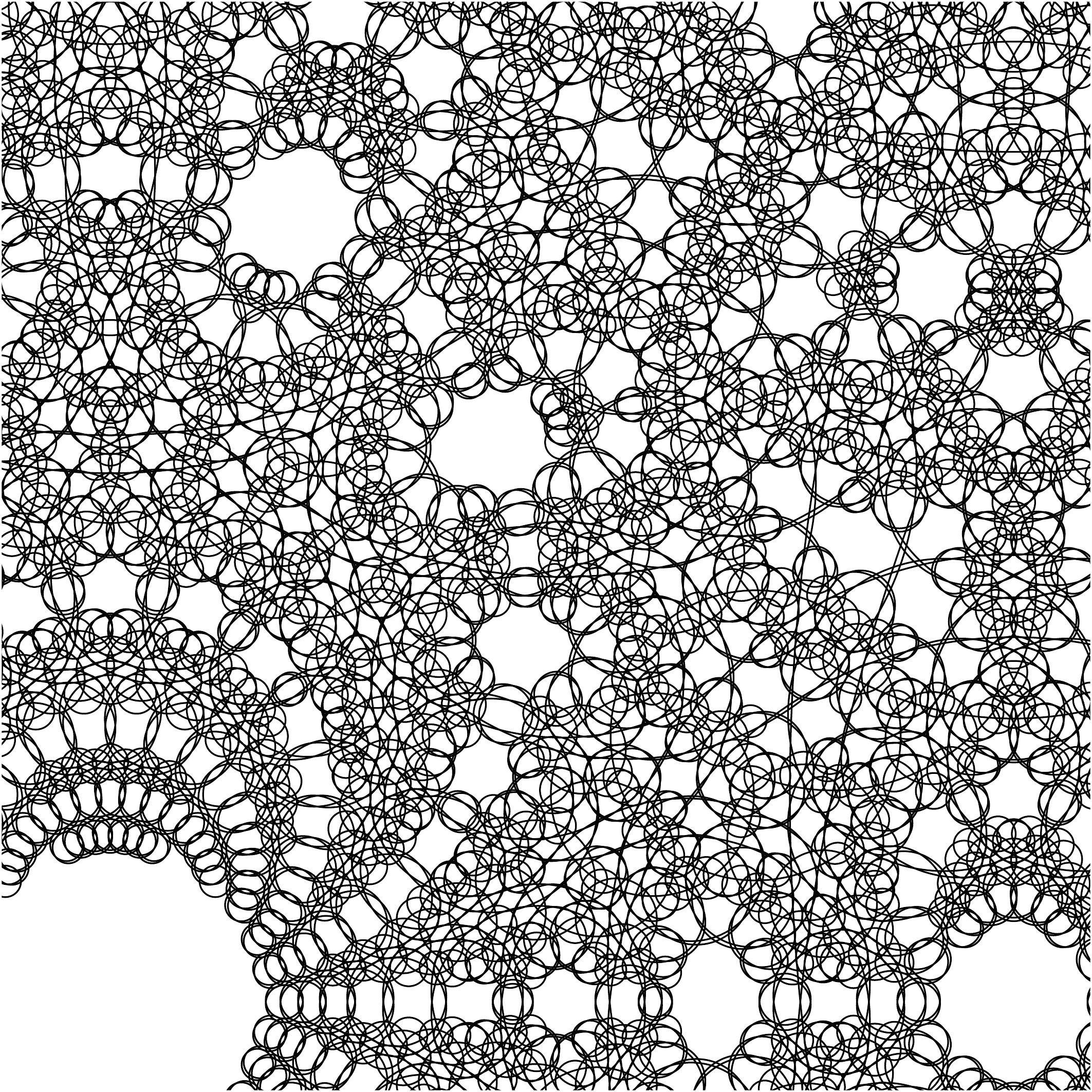}
    \captionsetup{width=1\textwidth}\caption{Piece of $\cS_1$ (left) and $\cS_{24}$ (right) for $\bQ(i\sqrt{19})$; curvatures up to $500/\sqrt{D}$.}
    \label{fig:1}
\end{figure}

Several cases have appeared in print before. Most notably, $\cS_{1}$ always contains (and \textit{equals} in Figure \ref{fig:1} as $19$ is prime) Stange's Schmidt arrangement \cite{stange}, which motivated this paper. Also, $\cS_{1}$ for $\bQ(i)$ is in \cite{graham}; $\cS_{2}$ for $\bQ(i\sqrt{2})$ is in \cite{guettler}; $\cS_{D}$ for $\bQ(i\sqrt{D})$ contains circle packings in \cite{boyd}, \cite{manna}, and \cite{nakamura} ($D=2$), \cite{soddy} ($D=3$), \cite{nakamura} ($D=6$), and \cite{baragar} (all $D\in\bN$); $\cS_{24}$ for $\bQ(i\sqrt{30})$ contains a circle packing in \cite{kapovich}; and $\cS_{D}$ for $D=(\Delta^2+14\Delta+1)/16$ when $4\nmid\Delta$ and $D=(\Delta^2+12\Delta)/16$ when $4\,|\,\Delta$ are Stange's ghost circles \cite{stange}.

Background is covered in Section \ref{sec:2}, then $\cS_D$ is introduced formally in Definition \ref{def:SD}. The main result of Section \ref{sec:3} is Theorem \ref{thm:redef}, which splits our arrangements into two very different categories. Figure \ref{fig:1} shows one arrangement from each. Those in the same category as $\cS_1$ (first image) are ``orbits" under subsets of $\text{PGL}_2(K)$---a fact used throughout the rest of the paper to connect the geometry of $\cS_D$ to the number theory of $K$. All aforementioned arrangements from literature belong to this category except the packing in \cite{kapovich} and certain ghost circles \cite{stange}.

Section \ref{sec:4} begins with the result below. It shows the extent to which $\cS_D$ is universal, explaining why our paper is a common thread among many past works. \textit{Integral} means that all curvatures are integers up to a single scaling factor.

\begin{theorem}\label{thm:intro1}Any integral arrangement that is fixed by a Zariski dense subgroup of $\emph{PSL}_2(\bC)$ is contained in some $\cS_D$ after scaling, rotating, and translating.\end{theorem}

Consequences of Theorem \ref{thm:intro1} include a restriction on possible intersection angles within certain subarrangements called $\textit{bugs}$ from \cite{kapovich} (Corollary \ref{cor:angles}), a criterion for determining superintegrality (Definition \ref{def:super}, Proposition \ref{prop:supint}), and a geometric connection between $\cS_D$ and the local-global principle for curvatures in an integral circle packing (Theorem \ref{thm:local} and Corollary \ref{cor:local}).

\begin{figure}[t]
    \centering
    \includegraphics[width=0.66\textwidth,clip,trim=0.2cm 0.2cm 0.2cm 0.2cm]{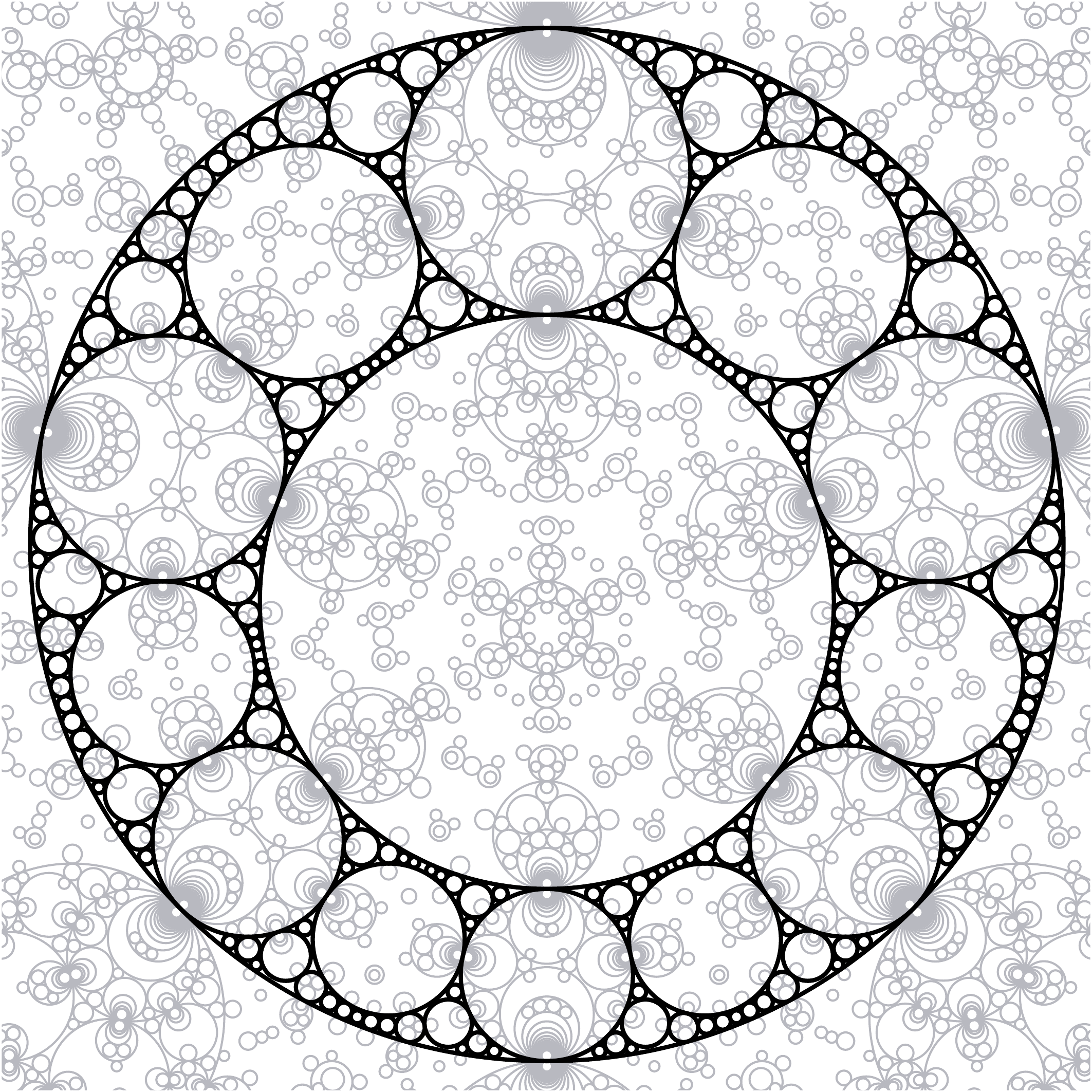}
    \captionsetup{width=0.66\textwidth}
    \caption{A new circle packing from $\cS_9$ for $\bQ(i\sqrt{39})$ that satisfies the local-global principal up to density 1.}\label{fig:2}
\end{figure}

Regarding the last result, the relative position of a circle packing in $\cS_D$ can give a sufficient condition for applying a theorem of Fuchs, Stange, and Zhang \cite{zhang}, which states that curvatures in certain packings have asymptotic density $1$ among integers that pass a set of local obstructions. A circle packing that satisfies our sufficient condition is displayed in Figure \ref{fig:2}. For any oriented circle in bold, at each of its intersection points in $\cS_9$ it is the largest exteriorly tangent circle that belongs to the packing. Such circles are called immediately tangent in \cite{stange2}.

\begin{corollary}Immediate tangency packings (Definition \ref{def:itp}) in $\cS_D$ satisfy the local-global principal up to density 1.\end{corollary}

Section \ref{sec:5} proves connections between $\cS_D$ and the class group of $K$. If $\alpha_0\in K$ lies on a circle in $\cS_D$, then the same is true of each $\alpha\in K$ for which the fractional ideals $(\alpha,1)$ and $(\alpha_0,1)$ belong to the same ideal class. So it makes sense to talk about which ideal classes are \textit{covered} (Definition \ref{def:cover}) by $\cS_D$ and which are not. Those which are covered can often be distinguished by the geometry of $\cS_D$. In Figure \ref{fig:3}, for example, circles intersect non-tangentially at some $\alpha\in\bQ(i\sqrt{39})$ if and only if the ideal class of $(\alpha,1)$ generates the class group (consequence of Proposition \ref{thm:angles}).

\begin{figure}[t]
    \centering
    \includegraphics[height=0.64\textwidth,clip,trim=0cm 1.9cm 0cm 4.9cm]{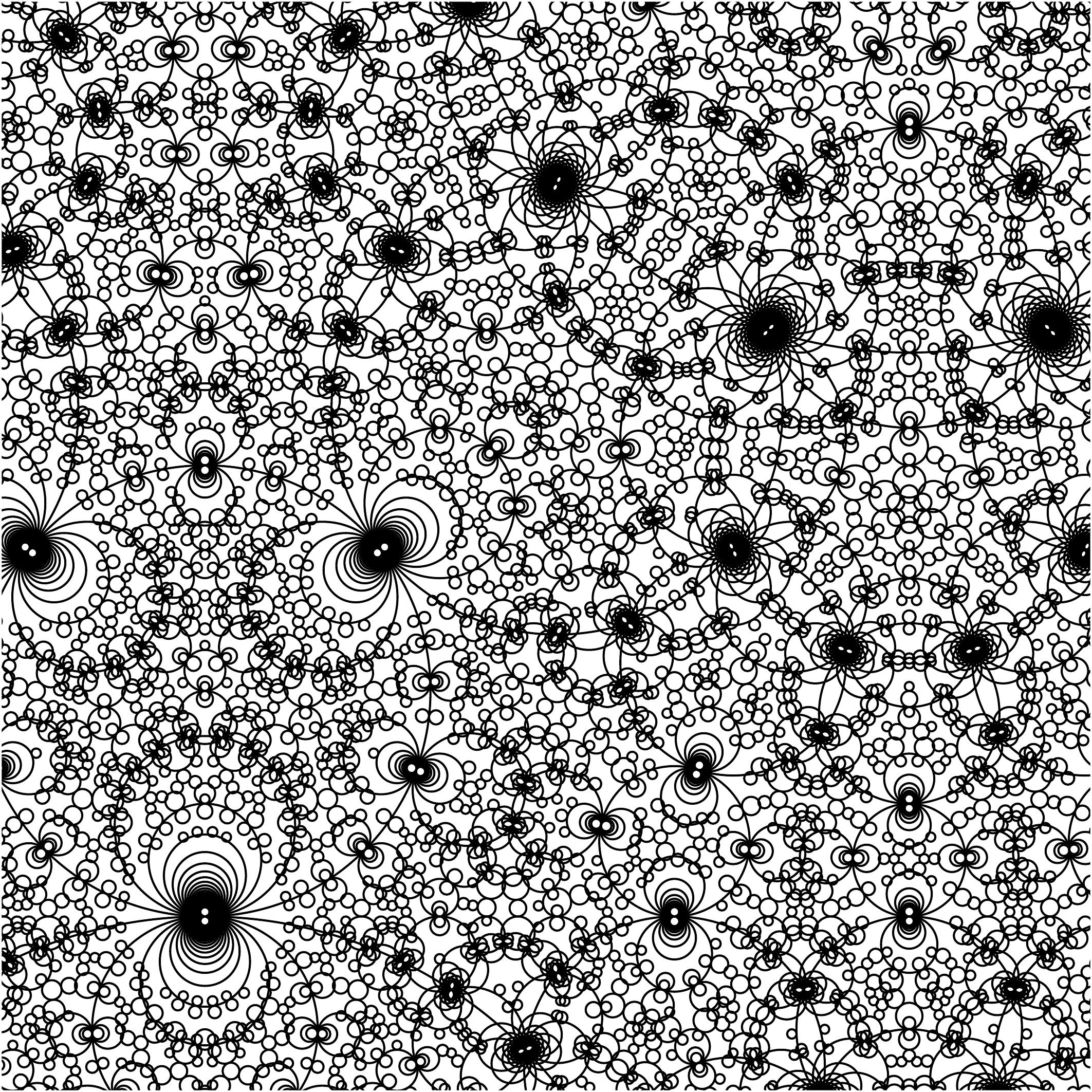}
    \captionsetup{width=0.936\textwidth}
    \caption{Piece of $\cS_4$ for $\bQ(i\sqrt{39})$, showing circles of curvature up to $150$.}\label{fig:3}
\end{figure}

Also in Figure \ref{fig:3}, note that any two circles appear to be linked by a chain of circles. This arrangement turns out to be \textit{finitely connected} (Definition \ref{def:fincon}), meaning there is an upper bound on the chain length needed to connect any circle of nonzero curvature to one at least twice as large. 

\begin{theorem}\label{thm:intro2}If $\cS_D$ is finitely connected, then it covers all of $\widehat{K}$. In particular, the class group of $K$ is generated by ideal classes of primes with norm dividing $D\Delta$.\end{theorem}

Primes dividing $\Delta$ generate the $2$-torsion subgroup of the class group, so primes dividing $D$ are responsible for each $2$-torsion coset in Theorem \ref{thm:intro2}. If we strengthen the connectivity hypothesis by assuming consecutive circles in a chain intersect at a point in $K$, then $D\Delta$ can be replaced by $D$ in the theorem (Corollary \ref{cor:finrat}). This is the case in Figure \ref{fig:3}. So $\cS_4$ is a kind of geometric assertion that the class group of $\bQ(i\sqrt{39})$ is cyclic, generated by a prime over $2$.

\section{Background}\label{sec:2}

As detailed in \cite{stange}, an \textit{oriented circle} $C$ is the set of $\alpha/\beta\in\widehat{\bC}$ solving \begin{equation}\label{eq:1}r|\alpha|^2-2\Re(\overline{\alpha}\beta\zeta)+\hat{r}|\beta|^2=0\end{equation} for some $\zeta\in\bC$ and $r,\hat{r}\in\bR$ satisfying $|\zeta|^2-\hat{r}r=1$. Given $C$, let $\bv_C=[\zeta\;\;\overline{\zeta}\;\;\hat{r}\;\;r]$.

If $r\neq 0$, (\ref{eq:1}) rearranges to the equation of a circle with center $\zeta/r$ and radius $1/|r|$: $|\alpha/\beta-\zeta/r|=1/|r|$. The sign of $r$ indicates orientation. Positive defines the \textit{interior} of $C$ so as to contain its center, while this is the \textit{exterior} when $r$ is negative. If $r=0$, then the point at infinity, $1/0$, is a solution, and $C$ appears as a line orthogonal to $\zeta$. In this case, \textit{interior} is defined as the side to which $\zeta$ points. 

The \textit{curvature} of $C$ is $r$, the \textit{curvature-center} is $\zeta$, and the \textit{cocurvature} is $\hat{r}$.

$\text{PSL}_2(\bC)$ acts on $\widehat{\bC}$ as M\"{o}bius transformations, which map oriented circles to oriented circles and preserve the orientation just described. This action corresponds to a linear transformation on vectors $\bv_C$. See \cite{carter} for a proof of the following.

\begin{proposition}\label{prop:exiso}If $M\in\emph{PSL}_2(\bC)$, then \begin{equation}\label{eq:2}M=\begin{bmatrix}\alpha & \gamma \\ \beta & \delta\end{bmatrix}\hspace{0.5cm}\text{and}\hspace{0.5cm}N=\begin{bmatrix}\alpha\overline{\delta} & \beta\overline{\gamma} & \alpha\overline{\gamma} & \beta\overline{\delta} \\ \overline{\beta}\gamma & \overline{\alpha}\delta & \overline{\alpha}\gamma &  \overline{\beta}\delta\\ \alpha\overline{\beta} & \overline{\alpha}\beta & |\alpha|^2 & |\beta|^2 \\ \gamma\overline{\delta} & \overline{\gamma}\delta & |\gamma|^2 & |\delta|^2\end{bmatrix}\end{equation} are such that $\bv_{MC}=\bv_CN$.\end{proposition}

The map $M\mapsto N$ is called the \textit{spin homomorphism}. Those $M$ for which $N$ has entries in $\cO$ form the \textit{extended Bianchi group}.

Vectors of the form $\bv_C$ generate a four-dimensional subspace of $\bC^2\times\bR^2$, which we identify with Minkowski space via the Hermitian form of signature $(1,3)$, $$Q=\frac{1}{2}\begin{bmatrix}1 & 0 & 0 & 0 \\ 0 & 1 & 0 & 0 \\ 0 & 0 & 0 & \!-1\! \\ 0 & 0 & \!-1 & 0\end{bmatrix}.$$ With respect to $Q$, each $\bv_C$ is a unit vector: $\|\bv_C\|_Q=\langle\bv_C,\bv_C\rangle_Q=|\zeta|^2-r\hat{r}=1$. In particular, $N$ from (\ref{eq:2}) preserves unit vectors (a so-called Lorentz transformation).

A modern perspective on the following properties is to treat them as inheritance from the geometry of Minkowski space. But they can also be verified directly with Euclidean geometry.

\begin{proposition}\label{prop:angle}$C$ and $C'$ intersect if and only if $\langle\bv_C,\bv_{C'}\rangle_Q\in[-1,1]$, in which case $\arccos(\langle\bv_C,\bv_{C'}\rangle_Q)$ is the intersection angle.\end{proposition}

\begin{proposition}\label{prop:reflect}The vector associated to the image of $C'$ reflected over $C$ is $\bv_{C'}-2\langle\bv_C,\bv_{C'}\rangle_Q\bv_C$. In terms of M\"{o}bius transformations, this image is $$\begin{bmatrix}\zeta & -\hat{r} \\ r & -\overline{\zeta}\end{bmatrix}\overline{C'},$$ where $[\zeta\;\,\overline{\zeta}\;\,\hat{r}\;\,r]=\bv_C$ and $\overline{C'}$ is the image of $C'$ under complex conjugation.\end{proposition}

\section{The object of study}\label{sec:3}

\begin{definition}\label{def:SD}Given a fixed imaginary quadratic field, for $D\in\bR$ let $\cS_D$ denote the set of oriented circles $C$ for which $i\sqrt{D}\bv_C\in\cO^2\times\sqrt{\Delta}\bZ^2$.\end{definition}

\begin{proposition}\label{prop:nonempty}$\cS_D$ is nonempty if and only if $D$ is a positive integer for which there exists $\alpha\in\cO$ satisfying $|\alpha|^2\equiv D\,\emph{mod}\,\Delta$.\end{proposition}

\begin{proof}Suppose the congruence holds for some $\alpha\in\cO$ and positive integer $D$. Then $$\bv=\left[\frac{i\alpha}{\sqrt{D}} \;\;\; \frac{\overline{i\alpha}}{\sqrt{D}} \;\;\;\frac{|\alpha|^2-D}{\sqrt{D|\Delta|}} \;\;\;\frac{\sqrt{|\Delta|}}{\sqrt{D}}\right]$$ corresponds to a circle in $\cS_D$ since $\|\bv\|_Q=1$, and $i\sqrt{D}\bv\in\cO^2\times\sqrt{\Delta}\bZ^2$.

Conversely, let $C\in\cS_D$ and let $\bv_C=[\zeta\;\;\overline{\zeta}\;\;\hat{r}\;\;r]$. Set $\alpha=i\sqrt{D}\zeta\in\cO$. Scaling $1=\|\bv_C\|_Q$ by $D$ gives $D=|\alpha|^2-D\hat{r}r\equiv|\alpha|^2\,\text{mod}\,\Delta$, where the congruence is due to $i\sqrt{D}r,i\sqrt{D}\hat{r}\in\sqrt{\Delta}\bZ$. This shows $D\in\bZ$, and $i\sqrt{D}r\in\sqrt{\Delta}\bZ$ shows $D>0$.\end{proof}

All arrangements are tacitly assumed to be nonempty for the rest of the paper.

\subsection{Basic properties}Here we determine the symmetries of $\cS_D$ as well as its geometry at points where circles intersect.

\begin{proposition}\label{prop:symmetry}The extended Bianchi group is the maximal subgroup of $\emph{PSL}_2(\bC)$ that fixes $\cS_D$.\end{proposition}

\begin{proof}Any subgroup of $\text{PSL}_2(\bC)$ that preserves $\cS_D$ is discrete since its image under the spin homomorphism preserves the lattice $\cO^2\times\sqrt{\Delta}\bZ^2$. So it suffices to show the extended Bianchi group is such a subgroup, as it is the maximal discrete subgroup of $\text{PSL}_2(\bC)$ containing $\text{PSL}_2(\cO)$ \cite{elstrodt}.

Let $C\in\cS_D$. Let $M$ belong to the extended Bianchi group so that $N$ from (\ref{eq:2}) has entries from $\cO$. Then the first two entries of $i\sqrt{D}\bv_CN$ also belong to $\cO$. Next, the third entry of $i\sqrt{D}\bv_CN$ is $$2i\Im(\alpha\overline{\gamma}(i\sqrt{D}\zeta))+(i\sqrt{D}\hat{r})|\alpha|^2+(i\sqrt{D}r)|\gamma|^2\in\sqrt{\Delta}\bZ.$$ The fourth entry takes a similar form and is also in $\sqrt{\Delta}\bZ$. Thus Proposition \ref{prop:exiso} gives $i\sqrt{D}\bv_{MC}=i\sqrt{D}\bv_CN\in \cO^2\times\sqrt{\Delta}\bZ^2$, implying $MC\in\cS_D$.\end{proof}

Affine transformations in the extended Bianchi group are translations by $\cO$ and rotations by the unit group $\cO^*$. Translative symmetry can be seen in the first image of Figure \ref{fig:4}, which shows $\cS_{8}$ centered on a fundamental region for the integers in $\bQ(i\sqrt{31})$. Some reflective symmetries generated by complex conjugation (which always fixes $\cS_D$), negation, and translation are also visible in Figure \ref{fig:4}.

\begin{corollary}$\cS_D$ is dense in $\widehat{\bC}$.\end{corollary}

\begin{proof}The orbit of any point under the extended Bianchi group is dense.\end{proof}

\begin{proposition}\label{prop:discrete}The angle between intersecting circles in $\cS_D$ is $\theta=\arccos(n/2D)$ for some $n\in\bZ$, where $n$ is even if $\Delta$ is. Moreover, the point(s) of intersection is/are in $\widehat{K}$ if and only if $e^{i\theta}\in K$.\end{proposition}

\begin{proof}If $i\sqrt{D}\bv_C,i\sqrt{D}\bv_{C'}\in\cO^2\times\sqrt{\Delta}\bZ^2$ then $\langle i\sqrt{D}\bv_C,i\sqrt{D}\bv_{C'}\rangle_Q=n/2$ for some $n\in\bZ$, where $n$ is even if $\Delta$ is. By Proposition \ref{prop:angle}, $C$ and $C'$ intersect if and only if $\langle\bv_C,\bv_{C'}\rangle_Q\in [-1,1]$, in which case the angle of intersection is $\arccos(n/2D)$.

For the second claim, suppose $C$ and $C'$ intersect (but do not coincide). Since $\text{PSL}_2(\cO)$ preserves $\cS_D$, intersection angles, and $\widehat{K}$, we may assume without loss of generality that $C$ and $C'$ have nonzero curvature $r$ and $r'$---if not replace them by $MC$ and $MC'$ for almost any choice of $M\in\text{PSL}_2(\cO)$. Let $\zeta$ and $\zeta'$ denote their curvature-centers. Then, with $\theta=\arccos(\langle\bv_C,\bv_{C'}\rangle_Q)\in[0,\pi]$, the point(s) of intersection is/are $(\zeta-\zeta'e^{\pm i\theta})/(r-r'e^{\pm i\theta})$. Scaling numerator and denominator by $i\sqrt{D}$ shows that intersections belongs to $K$ if and only if $e^{i\theta}$ does.\end{proof}

In $\cS_8$ for $\bQ(i\sqrt{31})$, shown in the first image of Figure \ref{fig:4}, there are four possible values of $\langle\bv_C,\bv_{C'}\rangle_Q$ that lie in $[-1,1]$. They are $\pm 1$, which correspond to tangential intersections, and $\pm 15/16$, which correspond to non-tangential intersections. Since $\arccos(\pm 1) = \text{arg}(\pm 1)$ and $\arccos(\pm 15/16) = \text{arg}((15\pm i\sqrt{31})/16)$, all points of intersection are in $\widehat{K}$ by Proposition \ref{prop:discrete}.

\subsection{A dichotomy}\label{ss:dich}This subsection splits our arrangements into two categories. The category to which $\cS_D$ belongs is determined by the Hilbert symbol $(D,\Delta)$. We modify usual notation to avoid confusion with ideals. 

\begin{notation}\label{not:hilbert}For $a,b\in\bQ$, let $H(a,b)\hspace{-0.02cm}=\hspace{-0.02cm}1$ if $ax^2+\hspace{0.03cm}by^2\hspace{-0.02cm}=\hspace{-0.02cm}z^2$ has a nonzero solution $x,y,z\in\bQ$, and let $H(a,b)=-1$ otherwise.\end{notation}

\begin{lemma}\label{lem:split}If any circle in $\cS_D$ contains a point in $\widehat{K}$ then $H(D,\Delta)=1$, and if $H(D,\Delta)=1$ then every circle in $\cS_D$ contains a point in $\widehat{K}$\end{lemma}

\begin{proof}Suppose $\alpha\in C\cap \widehat{K}$ for some $C\in\cS_D$ with curvature-center $\zeta$ and curvature $r$. Since $\text{PSL}_2(\cO)$ fixes $\widehat{K}$ and $\cS_D$, we may assume $r\neq 0$ without loss of generality (as in the previous proof). Then $|\alpha-\zeta/r|=1/|r|$, implying $i\sqrt{D}r\alpha-i\sqrt{D}\zeta$ is an element of $K$ with magnitude $\sqrt{D}$. Denote this element $(z+y\sqrt{\Delta})/x$ to see that $Dx^2+\Delta y^2=z^2$ is solvable. 

Now suppose $Dx^2+\Delta y^2=z^2$ for $x,y,z\in\bQ$, not all zero. Then $\Delta<0$ forces $x\neq 0$, so $\alpha=(z+y\sqrt{\Delta})/x\in K$ has magnitude $\sqrt{D}$. Thus if $C\in\cS_D$ has curvature-center $\zeta$ and curvature $r\neq 0$, then $\zeta/r+\alpha/i\sqrt{D}r\in C\cap \widehat{K}$.\end{proof}

The contrast between the two types of arrangements is displayed in Figure \ref{fig:1}, where $H(1,-19)=1$ and $H(24,-19)=-1$. By the lemma, $\cS_{24}$ avoids all points in $\bQ(i\sqrt{19})$. This makes it appear like the photographic negative of $\cS_1$, whose focal points are those in $\bQ(i\sqrt{19})$ (a consequence of Theorem \ref{thm:angles}).

We use $\widehat{\bR}$ to denote the extended real line with positive orientation, so its interior is the upper half-space. This way $\bv_{\widehat{\bR}}=[i\;-i\;\,0\;\,0]$. The following is proved in \cite{stange}.

\begin{lemma}\label{lem:stange}\emph{(Stange \cite{stange}) }For $M\in\emph{PSL}_2(\bC)$ with entries as in \emph{(\ref{eq:2})}, $M\widehat{\bR}$ has curvature-center $i(\alpha\overline{\delta}-\overline{\beta}\gamma)$, cocurvature $i(\alpha\overline{\gamma}-\overline{\alpha}\gamma)$, and curvature $i(\beta\overline{\delta}-\overline{\beta}\delta)$.\end{lemma}

\begin{notation}For $M\in\text{PGL}_2(K)$, let $(M)$ denote the ideal generated by its four entries, and let $\|(M)\|$ denote the ideal's norm.\end{notation}

\begin{lemma}\label{lem:coprime}Let $M\in\emph{PGL}_2(K)$ have entries as in \emph{(\ref{eq:2})}. For any $\lambda,\mu\in K$ we have $(\lambda,\mu)(\det M)/(M)\subseteq (\alpha\lambda + \gamma\mu,\beta\lambda + \delta\mu)\subseteq (\lambda,\mu)(M)$. In particular, if $N\in\emph{PSL}_2(\cO)$ then $(MN)=(M)=(NM)$.\end{lemma}

\begin{proof}Observe that $$(\lambda\det M,\mu\det M) = (\delta(\alpha\lambda + \gamma\mu)-\gamma(\beta\lambda + \delta\mu),\alpha(\beta\lambda+\delta\mu)-\beta(\alpha\lambda + \gamma\mu))\subseteq$$ $$(M)(\alpha\lambda + \gamma\mu,\beta\lambda + \delta\mu)\subseteq (\lambda,\mu)(M)^2.$$ Divide all sides by $(M)$ to get the desired containment.

To see that $(MN)=(M)$ for $N\in\text{PSL}_2(\cO)$, we will let $N$ play the role of ``$M$" from the previous argument. Let $\lambda$ and $\mu$ denote the top-row entries of $M$. Since $(\det N)/(N) = \cO = (N)$, we see that the ideal generated by the top-row entries of $MN$ is still $(\lambda,\mu)$. The ideal generated by bottom-row entries of $M$ is similarly preserved, giving $(MN)=(M)$. For $NM$ let $\lambda$ and $\mu$ be a column of $M$ instead.\end{proof}

\begin{theorem}\label{thm:redef}If $H(D,\Delta)=1$ then $C\in \cS_D$ if and only if $C=M\widehat{\bR}$ for some $M\in\emph{PGL}_2(K)$ with $|\det M|/\|(M)\|=\sqrt{D}$.\end{theorem}

\begin{proof}Let $M\in\text{PGL}_2(K)$ have entries as in (\ref{eq:2}), and suppose $|\det M|/\|(M)\|=\sqrt{D}$ . We claim $C=M\widehat{\bR}\in\cS_D$. Since $M/\sqrt{\det M}\in\text{PSL}_2(\bC)$, the formulas in Lemma \ref{lem:stange} show that $i\sqrt{D}\bv_C$ has entries $$i\sqrt{D}\zeta=i\sqrt{D}\left(\frac{i(\alpha\overline{\delta}-\overline{\beta}\gamma)}{|\det M|}\right)=\frac{i|\det M|}{\|(M)\|}\left(\frac{i(\alpha\overline{\delta}-\overline{\beta}\gamma)}{|\det M|}\right)=\frac{\overline{\beta}\gamma-\alpha\overline{\delta}}{\|(M)\|}\in\cO,$$ (and by similar arithmetic) $$i\sqrt{D}\hat{r}=\frac{\overline{\alpha}\gamma-\alpha\overline{\gamma}}{\|(M)\|}\in\sqrt{\Delta}\bZ,\hspace{0.5cm}\text{and}\hspace{0.5cm}i\sqrt{D}r=\frac{\overline{\beta}\delta-\beta\overline{\delta}}{\|(M)\|}\in\sqrt{\Delta}\bZ.$$ Thus $C\in\cS_D$ as claimed.

Now assume $H(D,\Delta)=1$. For $C\in\cS_D$ we seek $M\in\text{PGL}_2(K)$ with $M\widehat{\bR}=C$ and $|\det M|/\|(M)\|=\sqrt{D}$. By Lemma \ref{lem:split}, there exists $\alpha\in C\cap K$. Pick any split prime ideal $\fp$ that belongs to the same ideal class as $(\alpha,1)$ and does not contain $2D$. Let $p=\|\fp\|$, and pick $\alpha'\in\cO$ for which $(\alpha',p)=\fp$. Now take any $N\in\text{PSL}_2(\cO)$ such that $N(\alpha/1)=\alpha'/p$, which is possible since the ideal classes of $(\alpha,1)$ and $(\alpha',p)$ are equal. Our strategy is to find $M'\in\text{PGL}_2(K)$ with $|\det M'|/\|(M')\|=\sqrt{D}$ and $M'\widehat{\bR}=NC$. Then $M=N^{-1}M'$ is the desired matrix by Lemma \ref{lem:coprime}.

Let $\zeta$ and $r$ denote the curvature-center and curvature of $NC$. Fix any $\beta\in\fp$ with $\Im(\beta)=\sqrt{|\Delta|}/2$. Set $\gamma'=i\sqrt{D}\zeta+(i\sqrt{D}r/\sqrt{\Delta})(\alpha'\overline{\beta}/p)$ and $\delta'=(i\sqrt{D}r/\sqrt{\Delta})\beta$, both of which are in $\cO$. Define $$M'=\begin{bmatrix}\alpha' & \gamma'\\ p & \delta'\end{bmatrix}.$$ Substituting in the formulas for $\gamma'$ and $\delta'$ gives $$|\det M'|=\left|\frac{i\sqrt{D}r\alpha'(\beta-\overline{\beta})}{\sqrt{\Delta}}-i\sqrt{D}\zeta p\right|=\sqrt{D}|r|p\left|\frac{\alpha'}{p}-\frac{\zeta}{r}\right|=\sqrt{D}p,$$ where the last equality uses $\alpha'/p\in NC$. From here it is straightforward to verify that $i(\alpha'\overline{\delta'}-p\gamma')/|\det M'|=\zeta$ and $i(p\overline{\delta'}-p\delta')/|\det M'|=r$. So $M'\widehat{\bR}=NC$. 

The proof will be complete if we show $(M')=\fp$, since then $|\det M'|/\|(M')\|=|\sqrt{D}p|/\|\fp\|=\sqrt{D}$. To this end, first observe that $\alpha'\not\in\overline{\fp}$ because $\fp\neq\overline{\fp}$ and $p\nmid \alpha'$. Next we claim that $p\nmid i\sqrt{D}r$, which would imply $\delta'\not\in\overline{\fp}$ by choice of $\beta$. Indeed, if $p\,|\,i\sqrt{D}r$ then $\alpha'/p\in NC$ gives $$0=i\sqrt{D}r|\alpha'|^2 + 2\Im(\overline{\alpha'}p(i\sqrt{D}\zeta)) + i\sqrt{D}\hat{r}p^2\equiv 2\Im(\overline{\alpha'}p(i\sqrt{D}\zeta))\,\text{mod}\,p^2.$$ But then $i\sqrt{D}\zeta\in\fp$ is forced, in turn showing $p$ divides $|i\sqrt{D}\zeta|^2+(i\sqrt{D}\hat{r})(i\sqrt{D} r)=D$, contradicting our choice of $\fp$. So $\delta'\not\in\overline{\fp}$ as desired. Now, $\alpha'\delta'\not\in\overline{\fp}$ gives $\det M'=\alpha'\delta'-p\gamma'\not\in\overline{\fp}$. Combined with $p^2\,|\,|\det M'|^2$, this implies $\det M'\in\fp^2$. But then $\alpha',\delta'\in\fp$ and $(p,\fp^2)=\fp$ forces $\gamma'\in\fp$, completing the proof.\end{proof}

The matrices used in Theorem \ref{thm:redef}, $M\in\text{PGL}_2(K)$ with $|\det M|/\|(M)\|=\sqrt{D}$, form the extended Bianchi group when $D=1$ (though typically each matrix $M$ is scaled by $1/\sqrt{\det M}$ so as to obtain a subgroup of $\text{PSL}_2(\bC)$). In particular, Theorem \ref{thm:redef} asserts that the extended Bianchi group acts transitively on $\cS_1$. More generally, it can be shown as a corollary to the theorem that when $H(D,\Delta)=1$, the extended Bianchi group acts transitively on $\cS_D$ if and only if $D$ is square-free. The author believes the same is true when $H(D,\Delta)=-1$, but has not proved it.

\section{Relation to subarrangements found in literature}\label{sec:4}

\subsection{Universality} We would like to determine what kinds of arrangements can occur as subarrangements of some $\cS_D$. Integrality of curvatures (or cocurvatures or curvature-centers) up to a single scaling factor is evidently necessary, but not sufficient to guarantee containment in some $\cS_D$. A certain amount of symmetry is also required to avoid random collections of integral circles that are not found in any $\cS_D$. The Kapovich-Kontorovich ``subarithmeticity theorem" suggests what sufficient symmetry might be.

\begin{theorem}\label{thm:kap}\emph{(Special case of Kapovich-Kontorovich \cite{kapovich})} If the orbit of a circle under a Zariski dense subgroup $\Gamma<\emph{PSL}_2(\bC)$ is integral, then $\Gamma$ is contained in a group weakly commensuarable to some $\emph{PSL}_2(\cO)$.\end{theorem}

So Zariski denseness of $\Gamma$ is enough to associate an imaginary quadratic field to an integral orbit of a circle. Weak commensurability means there exists $M\in\text{PSL}_2(\bC)$ such that $M\Gamma M^{-1}\cap\text{PSL}_2(\cO)$ is a finite index subgroup of $M\Gamma M^{-1}$. Letting $C$ be a circle from our integral orbit, if it happens that $MC\in\cS_D$ for some $D$, then $(M\Gamma M^{-1}\cap\text{PSL}_2(\cO))MC\subseteq\cS_D$ by Proposition \ref{prop:symmetry}. But there seems no reason that this should be true, or that it should be true for any remaining circles in $M\Gamma C$. And even if all of $M\Gamma C$ is contained in some $\cS_D$, why should an integral arrangement composed of multiple orbits be contained in a single $\cS_D$? The next theorem addresses these concerns.

\begin{theorem}\label{thm:univ}Any integral arrangement that is fixed by a Zariski dense subgroup of $\emph{PSL}_2(\bC)$ is contained in some $\cS_D$ after scaling, rotating, and translating.\end{theorem}

\begin{proof}Let $\cA$ denote the arrangement, scaled to have curvatures in $\bZ$, and let $\Gamma$ be the maximal subgroup of $\text{PSL}_2(\bC)$ that fixes $\cA$. Theorem \ref{thm:kap} gives a finite index subgroup $\Gamma_0\leq \Gamma$ that is, up to conjugation, contained in some $\text{PSL}_2(\cO)$. Since conjugation preserves traces, all traces from $\Gamma_0$ belong to $\cO$. Moreover, $\Gamma_0$ has finite index in a Zariski dense group and is thus Zariski dense itself. In particular, the invariant trace field of $\Gamma_0$, which is generated over $\bQ$ by $\tr\,M^2$ for $M\in\Gamma_0$, must be nonreal \cite{maclachlan}. So fix some $M_0\in\Gamma_0$ with $\Im(\tr\,M_0^2)\neq 0$. Let $\tau=\tr\,M_0\in\cO$ and set $t=(\tau^2-\overline{\tau^2})/\sqrt{\Delta}=2\Im(\tr\,M_0^2)/\sqrt{|\Delta|}\in\bZ$.

Let $\alpha$, $\beta$, $\gamma$, and $\delta$ be the entries of $M_0$ as in (\ref{eq:2}). Suppose for a contradiction that $\beta = 0$. Since $\Gamma\cA=\cA$ must be Zariski dense in $\text{PSL}_2(\bC)\cA$, there is some $C\in \cA$ with nonzero curvature $r$. From the fourth column of $N$ in (\ref{eq:2}) with $\beta$ set to $0$, we see that the curvature of $M_0^nC$ for $n\in\bZ$ is $|\delta|^{2n}r$, which is assumed to be integral for any $n$. Thus $|\delta| = 1$, which combines with $1 = \det M = \alpha\delta$ to give $\alpha=\overline{\delta}$. Finally, $\tr\,M_0^2=\alpha^2 + \delta^2 = \overline{\delta^2}+\delta^2$ contradicts $\Im(\tr\,M_0^2)\neq 0$. Thus $\beta\neq 0$.

We claim that after translating $\cA$ by $-\alpha/\beta$ and rotating/scaling by $t\beta$, it is contained in $\cS_D$ with \begin{equation}\label{eq:3}D=\big|t\beta\sqrt{\Delta}\big|^2.\end{equation} First observe that the new copy of $\cA$ has curvatures in $\bZ/|t\beta|$, and its symmetry group contains \begin{equation}\label{eq:4}\begin{bmatrix}t\beta & 0 \\ 0 & 1\end{bmatrix}\begin{bmatrix}1 & -\alpha/\beta \\ 0 & 1\end{bmatrix}M_0\begin{bmatrix}1 & \alpha/\beta \\ 0 & 1\end{bmatrix}\begin{bmatrix}1/t\beta & 0 \\ 0 & 1\end{bmatrix}=\begin{bmatrix}0 & -t \\ 1/t  & \tau\end{bmatrix}.\end{equation} Relabel $M_0$ as the last matrix above. 

Let $R_0$ denote the unique matrix satisfying $\bv_CR_0=[r_{-2}\;\,r_{-1}\;\,r_0\;\,r_1]$ for any oriented circle $C$, where $r_n$ denotes the curvature of $M_0^nC$. The $n^{\text{th}}$ column of $R_0$ can be found by using the bottom two entries of $M_0^{n-3}$ instead of $\beta$ and $\delta$ in the fourth column of $N$ from (\ref{eq:2}). This gives  $$R_0=\begin{bmatrix}\tau/t & 0 & 0 & \overline{\tau}/t\\ \overline{\tau}/t & 0 & 0 & \tau/t \\ |\tau/t|^2 & 1/t^2 & 0 & 1/t^2\\ 1 & 0 & 1 & |\tau|^2\end{bmatrix}.$$ We have $\det R_0=(\tau^2-\overline{\tau^2})/t^4=\sqrt{\Delta}/t^3\neq 0$. Letting $\sigma=\overline{\tau}(1-\tau^2)\in\cO$, the first and third columns of $R_0^{-1}$ are $$\begin{bmatrix}\tau/\sqrt{\Delta}\\ \sigma/\sqrt{\Delta} \\ -\overline{\sigma}/\sqrt{\Delta} \\ -\overline{\tau}/\sqrt{\Delta}\end{bmatrix}\hspace{0.5cm}\text{and}\hspace{0.5cm}\begin{bmatrix}0\\ t^2 \\ 0 \\ 0\end{bmatrix}.$$ 

Recall that $r_n\in\bZ/|t\beta|=\sqrt{\Delta}\bZ/i\sqrt{D}$ for any $C\in\cA$. So it follows from $\bv_C=[r_{-2}\;\,r_{-1}\;\,r_0\;\,r_1]R_0^{-1}$ that $C$ has curvature-center $(\tau r_{-2}+\sigma r_{-1}-\overline{\sigma}r_0-\overline{\tau}r_1)/\sqrt{\Delta}\in\cO/i\sqrt{D}$ and cocurvature $t^2r_{-1}\in\sqrt{\Delta}\bZ/i\sqrt{D}$. Thus $i\sqrt{D}\bv_C\in\cO^2\times\sqrt{\Delta}\bZ^2$.\end{proof}

The conjugating matrix product in (\ref{eq:4}) works as ``$M$" from the discussion immediately preceding the theorem. That is, our affine transformation $M$ gives $M\Gamma M^{-1}\cap\text{PSL}_2(\cO)$ finite index in $M\Gamma M^{-1}$. We do not need this fact, so we do not prove it.

In \cite{kapovich}, Kapovich and Kontorovich introduce \textit{bugs} as a generalization of circle packings. A bug is a set of oriented circles in which overlapping interiors must intersect at angle $\pi/n$ for $n$ belonging to some finite subset of $\bN$.

\begin{corollary}\label{cor:angles}If $\cA$ is integral with Zariski dense symmetry group, the angle $\theta$ between two intersecting circles of $\cA$ satisfies $\cos(\theta)\in\bQ$. In particular, if $\cA$ is a bug then $\theta$ is either $0$, $\pm \pi/3$, or $\pi/2$.\end{corollary}

\begin{proof}The first claim is a combination of Proposition \ref{prop:discrete} and Theorem \ref{thm:univ}. The claim about bugs is Niven's theorem \cite{niven}: $\arccos(\pi/n)\in\bQ$ only if $n\in\{1,2,3\}$.\end{proof}

As the remainder of this section shows, actually identifying an arrangement $\cA$ within some $\cS_D$ is more useful than just knowing it is possible. Note that the proof of Theorem \ref{thm:univ} gives a formula for $D$ in (\ref{eq:3}). 

\subsection{Superintegrality}Our first property of subarrangements of $\cS_D$ applies when $D$ divides $\Delta$. The reader may recall from the introduction that this happens to be true of many, perhaps most, arrangements already appearing in literature.

The Apollonian supergroup is introduced in \cite{graham} by Graham, Lagarias, Mallows, Wilks, and Yan. Building on their idea, Kontorovich and Nakamura consider the supergroup more generally. We use their definitions \cite{nakamura}.

\begin{definition}\label{def:super}The \textit{supergroup} of an arrangement $\cA$ is generated by reflections over its circles as well as the matrices in $\text{PSL}_2(\bC)$ that fix $\cA$. We call $\cA$ \textit{superintegral} if the orbit under its supergroup is integral.\end{definition}

Superintegrality is used in \cite{chait} and \cite{nakamura} to help classify crystallographic circle packings, and again in \cite{kapovich} to help classify Kleinian circle packings and bugs. 

\begin{lemma}\label{lem:supint}If $\cA$ is integral with Zariski dense supergroup, then it is superintegral if and only if $2\langle \bv_C,\bv_{C'}\rangle_Q\in\bZ$ for all $C,C'\in\cA$. In this case, each intersection angle is either $0$, $\pm \pi/3$, or $\pi/2$.\end{lemma}

\begin{proof}Given oriented circles $C$ and $C'$, the reflection of $C'$ over $C$ corresponds to the vector $\bv_{C'}-2\langle\bv_C,\bv_{C'}\rangle_Q \bv_C$ by to Proposition \ref{prop:reflect}. Thus if $2\langle\bv_C,\bv_{C'}\rangle_Q\in\bZ$ for all $C,C'\in\cA$, reflections over circles from $\cA$ preserve the lattice generated by their vectors. This lattice is thus preserved by the supergroup's image under the spin homomorphism, implying curvatures in the orbit of the supergroup remain integral.

Conversely, suppose $\cA$ is superintegral. Let $M,M'\in\text{PSL}_2(\bC)$ be such that $C=M\widehat{\bR}$ and $C'=M'\widehat{\bR}$ belong to $\cA$. Let Let $\bv_C=[\zeta\;\,\overline{\zeta}\;\,\hat{r}\;\,r]$ and $\bv_{C'}=[\zeta'\;\,\overline{\zeta'}\;\,\hat{r}'\;\,r']$. By Proposition \ref{prop:reflect}, composing the reflection over $C$ followed by the reflection over $C'$ is given as a M\"{o}bius transformation by $$\begin{bmatrix}\zeta' & -\hat{r}' \\ r' & -\overline{\zeta'}\end{bmatrix}\begin{bmatrix}\overline{\zeta} & -\hat{r} \\ r & -\zeta\end{bmatrix}=\begin{bmatrix}\overline{\zeta}\zeta'-r\hat{r}' & \hat{r}'\zeta-\hat{r}\zeta' \\ r'\overline{\zeta}-r\overline{\zeta'} & \zeta\overline{\zeta'}-\hat{r}r'\end{bmatrix}.$$ Call the last matrix $N$, and observe that $\tr\,N=2\langle \bv_C,\bv_{C'}\rangle_Q$. Theorem \ref{thm:kap} says that some power of $N$, say $N^n$, is contained in $\text{PSL}_2(\cO)$ up to conjugation. Therefore $\tr\,N^n\in\cO$ because conjugation preserves traces. But $\tr\,N^n$ can be expressed as a monic, integral polynomial in $\tr\,N$, so we must have $\tr\,N\in\cO\cap\bR=\bZ$.

The final assertion about intersection angles follows from Proposition \ref{prop:angle}.\end{proof}

\begin{proposition}\label{prop:supint}If $D\,|\,\Delta$, all subsets of $\cS_D$ are superintegral.\end{proposition}

\begin{proof}Let $C,C'\in\cS_D$ and let $\bv_C$ and $\bv_{C'}$ be as in the last proof. Since $D\,|\,\Delta$ and $\Delta\,|\,(i\sqrt{D}\hat{r})(i\sqrt{D}r)$, $$D=|i\sqrt{D}\zeta|^2-D\hat{r}r\equiv|i\sqrt{D}\zeta|^2\,\text{mod}\,D.$$ Similarly, $D\,|\,|i\sqrt{D}\zeta'|^2$. As any prime ideal containing $\Delta$ equals its conjugate, we see that $D$ also divides $(i\sqrt{D}\zeta)(i\sqrt{D}\overline{\zeta'})$ and $(i\sqrt{D}\overline{\zeta})(i\sqrt{D}\zeta')$. Thus $2\langle\bv_C,\bv_{C'}\rangle_Q = \zeta\overline{\zeta'}+\overline{\zeta}\zeta'-\hat{r}r'-r\hat{r}'\in\bZ$, proving our claim by Lemma \ref{lem:supint}.\end{proof}

Baragar and Lautzenheiser recently discovered a one-parameter family of circle packings that generalized the classical Apollonian strip packing \cite{baragar}. They begin with four circles, say $C_1$, $C_2$, $C_3$, and $C_4$, oriented so as to have disjoint interiors. Two are parallel lines distanced $1$ apart. Two are circles of radius $1/2$ with centers distanced $\sqrt{D}$ apart for some $D\in\bN$, each circle tangent to both lines. It is then proved in \cite{baragar} that circles from the lattice generated by $\bv_{C_1}$, $\bv_{C_2}$, $\bv_{C_3}$, and $\bv_{C_4}$ are dense in $\widehat{\bC}$ and only intersect tangentially. Among positively oriented circles from this lattice, they keep those which lie between the two lines and are not contained in the interior of another positively oriented circle. This produces a circle packing. The same procedure is used to construct the circle packing in Figure \ref{fig:4} from its background arrangement, so the reader can see the resulting effect.

\begin{proposition}\label{prop:baragar}The Baragar-Lautzenheiser packing of parameter $D$ is contained in $\cS_D$ for $\bQ(i\sqrt{D})$.\end{proposition}

\begin{proof}Let $C_1$ be the vertical line through $0$ and $C_2$ the vertical line through $1$. Orient them to have disjoint interiors, so \begin{equation}\label{eq:5}\bv_{C_1}=[-1\;-\!1\;\,0\;\,0],\hspace{0.5cm}\text{and}\hspace{0.5cm}\bv_{C_2}=[1\;\,1\;2\;\,0].\end{equation} We also have two positively oriented circles of curvature $2$, $C_3$ and $C_4$: \begin{equation}\label{eq:6}\bv_{C_3}=[1\;\,1\;\,0\;\,2],\hspace{0.5cm}\text{and}\hspace{0.5cm}\bv_{C_4}=[1+2i\sqrt{D}\;\;\,1-2i\sqrt{D}\;\;\,2D\;\,2].\end{equation} They are tangent to $C_1$ and $C_2$ with centers distanced $\sqrt{D}$ apart. For $k=1,2,3,4$ we have $i\sqrt{D}\bv_{C_k}\in \bZ[i\sqrt{D}]^2\times i\sqrt{D}\bZ^2$ implying $C_k\in\cS_D$ for $K=\bQ(i\sqrt{D})$.\end{proof}

Note that if $D$ is divisible by a a perfect square, say $f^2\neq 1,4$, then $D$ does not divide the discriminant of $\bQ(i\sqrt{D})$. In these cases Proposition \ref{prop:supint} does not apply. Rather, $\bv_C$ for $C$ in the Baragar-Lautzenheiser packing satisfies $i\sqrt{D}\bv_C\in \cO_f^2\times \sqrt{\Delta_f}\bZ^2$, where $\cO_f$ is the order of discrimant $\Delta_f=f^2\Delta$. We avoid discussing non-maximal orders here, except to mention that the proof of Proposition \ref{prop:supint} still works for $D\,|\,\Delta_f$. Alternatively, Lemma \ref{lem:supint} can be applied directly to a Baragar-Lautzenheiser packing by checking that products of lattice generators in (\ref{eq:5}) and (\ref{eq:6}) are integers or half-integers. Either way, we see that Baragar-Lautzenheiser packings are superintegral.

\subsection{The almost local-global principle}Certain integral circle packings have congruence restrictions on what integers can appear as curvatures. For example, curvatures in classical Apollonian packings always lie in a proper subset of the congruence classes $\text{mod}\,24$ \cite{fuchs}. Graham, Lagarias, Mallows, Wilkes, and Yan conjectured (for Apollonian packings) that every sufficiently large integer which might occur as a curvature does \cite{lagarias}. Although this is still unresolved, Bourgain and Kontorovich showed that curvatures in an Apollonian packing have asymptotic density $1$ among integers passing local obstructions \cite{bourgain}.

In each of the packings considered here, local obstructions to curvatures always form a finite set of forbidden congruence classes modulo some positive integer $L_0$, as in the Apollonian case where $L_0=24$. This is a consequence of strong approximation for Zariski dense subgroups of $\text{PSL}_2(\cO)$ \cite{rapinchuk}. As such, we define the almost-local global principal \textit{with respect to $L_0$}. A more general definition need not assume $L_0$ exists.

\begin{definition}\label{def:local}An arrangement $\cA$ satisfies the \emph{almost local-global principle with respect to} $L_0$ if for some $\eps > 0$ and all sufficiently large $x$, the number of positive integers less than $x$ that occur as curvatures in $\cA$ is at least $kx/L_0-x^{1-\eps}$, where $k$ is the number of congruence classes $\text{mod}\,L_0$ represented by curvatures in $\cA$.\end{definition}

The Bourgain-Kontorovich result has recently been generalized by Fuchs, Stange, and Zhang.

\begin{theorem}\label{thm:stange}\emph{(Fuchs-Stange-Zhang \cite{zhang})} Suppose $\Gamma\leq\emph{PSL}_2(\cO)$ is finitely generated and Zariski dense, has infinite covolume, and contains a congruence subgroup of $\emph{PSL}_2(\bZ)$. If $M,N\in\emph{PGL}_2(K)$ with $N\widehat{\bR}$ tangent to $\widehat{\bR}$, then $M\Gamma N\widehat{\bR}$ satisfies the almost local-global principle with respect to some $L_0$ that depends only on $\Gamma$, $N$, and $|\det M|/\|(M)\|$.\end{theorem}

Remark that \cite{zhang} restricts attention to $M,N\in\text{PSL}_2(K)$, not $\text{PGL}_2(K)$ as stated above. Communication with the second- and third-named authors confirmed that their proof still applies. When $M\in\text{PSL}_2(K)$, a common denominator for the entries of $M$ is used in \cite{zhang} as a scaling factor to achieve integrality of certain quadratic forms. When $M\in\text{PGL}_2(K)$, the role of a common denominator squared can be assumed more generally by the integral ideal $(\det M)/(M)^2$.

The hypotheses of Theorem \ref{thm:stange} can often be verified without ever computing the symmetry group of an arrangement. Let us show how the manner in which a subarrangement sits in $\cS_D$ can be used to conclude the almost local-global principle. 

\begin{lemma}\label{lem:orbits}If $H(D,\Delta)=1$, then $\cS_D$ is a finite union of $\emph{PSL}_2(\cO)$ orbits.\end{lemma}

\begin{proof}By Theorem \ref{thm:redef}, every element of $\cS_D$ is of the form $M\widehat{\bR}$ for some $M\in\text{PGL}_2(K)$ with $|\det M|/\|(M)\|=\sqrt{D}$. If $M,M'\in\text{PGL}_2(K)$ are such that $(M)=(M')$, $\det M = \det M'$ and $M\equiv M'\,\text{mod}\,(\det M)/(M)^2$, then $M^{-1}M'\in\text{PSL}_2(\cO)$. That is, $M\widehat{\bR}$ and $M'\widehat{\bR}$ are in the same $\text{PSL}_2(\cO)$ orbit. There are finitely many possible ideals $(M)$ up to scaling, as well as finitely many integral ideals $(\det M)/(M)^2$ of norm $D$, each with only finitely many congruence classes of matrices.\end{proof}

The proof above overcounts the number of orbits, but it will not matter.

\begin{theorem}\label{thm:local}Let $\cA\subset\cS_D$ be such that each circle in $\cA$ intersects at least one other tangentially, and $\cA\cap M\cA$ is either empty or $\cA$ for any $M$ in some fixed congruence subgroup of $\emph{PSL}_2(\cO)$. Then $\cA$ satisfies the almost local-global principle.\end{theorem}

\begin{proof}Let $\Gamma_0$ be the maximal subgroup of $\text{PSL}_2(\cO)$ that preserves $\cA$. Let $\Gamma(L)\leq\text{PSL}_2(\cO)$ be the congruence subgroup (of level $L$) assumed in our hypothesis. 

Proposition \ref{prop:discrete} tells us tangential intersection in $\cS_D$ occurs only at points in $\widehat{K}$, which means $H(D,\Delta)=1$ by Lemma \ref{lem:split}. So each $C\in\cA$ is of the form $M\widehat{\bR}$ for some $M\in\text{PGL}_2(K)$ with $|\det M|/\|(M)\|=\sqrt{D}$. We claim for such a matrix $M$, that $MAM^{-1}\in\Gamma_0$ whenever $A\in\text{PSL}_2(\bZ)$ is in the principal congruence subgroup of level $DL$. To see this, first scale $M$ so that $\|(M)\|$ and $DL$ are coprime. Let $\text{adj}\,M=M^{-1}\det M$. We have \begin{equation}\label{eq:7}MA\,\text{adj}\,M\equiv M\,\text{adj}\,M= \begin{bmatrix}\det M & 0 \\ 0 & \det M\end{bmatrix}\,\text{mod}\,DL.\end{equation} In particular, the entries of $MA\,\text{adj}\,M$ are contained in $(\det M)/(M)^2$. But they are also contained in $(M)^2$, which is coprime to  $(\det M)/(M)^2$. Thus entries of $MA\,\text{adj}\,M$ are divisible by $\det M$. Now divide both sides of (\ref{eq:7}) by $\det M$, and recall that $\|(M)\|$ and $L$ are coprime to see that $MAM^{-1}\in\Gamma(L)$. Since $(MAM^{-1})C=(MAM^{-1})M\widehat{\bR}=M\widehat{\bR}=C$, we have $C\in \cA\cap (MAM^{-1})\cA$. So our hypothesis gives $MAM^{-1}\in\Gamma_0$ as claimed. Note that $C$ is the limit set of $MAM^{-1}$ for $A\in\text{PSL}_2(\bZ)$ from the principal congruence subgroup of level $DL$. In particular, $\cA$ is in the limit set of $\Gamma_0$, which is therefore Zariski dense.

Now, since $[\text{PSL}_2(\cO):\Gamma(L)]$ is finite, $\cS_D$ is a finite union of $\Gamma(L)$ orbits by Lemma \ref{lem:orbits}. But if it happens for some $C\in\cA$ and $M\in\Gamma(L)$ that $MC\in\cA$, then $MC\in\cA\cap M\cA$ implies $M\in\Gamma_0$ by our hypothesis. Therefore $\cA$ is a finite union of orbits of $\Gamma_0$. Fix one element from each orbit, call them $C_1,...,C_n$, as well as some $C_k'\in\cA$ tangent to $C_k$ for each $k=1,...,n$. Write $C_k'=M_k\widehat{\bR}$ for $M_k\in\text{PGL}_2(K)$ with $|\det M_k|/\|M_k\|=\sqrt{D}$. Since congruence subgroups of $\text{PSL}_2(\bZ)$ are finitely generated, for each $k=1,...,n$ we can find some $\Gamma_k\leq M_k^{-1}\Gamma_0M_k$ that is finitely generated and of infinite covolume, while still being Zariski dense and containing the principal congruence subgroup of level $DL$ in $\text{PSL}_2(\bZ)$. We may then write $\Gamma_0C_k$ as a union of ``orbits" of the form $(N_jM_k)\Gamma_k(M_k^{-1}C_k)$, where $N_j\in\text{PSL}_2(\cO)$ runs over a complete set of coset representatives for $\Gamma_0/(M_k\Gamma_kM_k^{-1})$.

Observe that $M_k^{-1}C_k$ is tangent to the real line by choice of $C_k'$, that $\Gamma_k$ meets the criteria of Theorem \ref{thm:stange}, and that $|\det(N_jM_k)|/\|(N_jM_k)\|=|\det M_k|/\|(M_k)\|=\sqrt{D}$ for all $j$ by Lemma \ref{lem:coprime}. Thus $(N_jM_k)\Gamma_k(M_k^{-1}C_k)$ satisfies the almost local-global principle with respect to some $L_k$ that does not depend on $j$. By taking any finite set $N_{j_1},...,N_{j_i}$ for which the curvatures of the corresponding orbits represent every congruence class $\text{mod}\,L_k$ represented by those in $\Gamma_0C_k$, we see that $\Gamma_0C_k$ also satisfies the almost local-global principle. Thus $\cA$ does as well, with respect to $L_0=\text{lcm}(L_1,...,L_n)$. \end{proof}

The argument in which we pass to the subgroup $\Gamma_k$ and observe that multiplying by coset representatives $N_j$ leaves $|\det M_k|/\|(M_k)\|$ unchanged can be used to show that the ``finitely generated" and ``infinite covolume" hypotheses are not needed in Theorem \ref{thm:stange}.

Theorem \ref{thm:local} has a nice geometric realization when we restrict attention to circle packings---an infinite arrangement in which circle interiors are dense and disjoint. In Figure \ref{fig:2} we constructed a packing from a single, initial oriented circle by taking its largest exterior neighbor at each point of tangency. Continuing in this way, just one circle uniquely determines the rest of the packing. In particular, the packing and its image under a matrix in $\text{PSL}_2(\cO)$ are either disjoint or equal as Theorem \ref{thm:local} requires. (The argument is treated formally in the next corollary). 

This construction does not produce circle packings in every $\cS_D$. The first image of Figure \ref{fig:4} shows $\cS_{8}$ for $\bQ(i\sqrt{31})$. Here every circle has two points of tangency where its largest exterior neighbors have overlapping interiors. To find packings in such cases, we search among proper subarrangements. The background circles in the second image are those from $\cS_8$ with curvature-center and cocurvature satisfying $i\sqrt{8}\zeta\in(2,(1-i\sqrt{31})/2)$ and $i\sqrt{8}\hat{r}\in (2)$. Then the strategy outlined in the previous paragraph creates the bolded circle packing.

\begin{figure}[t]
    \centering
    \includegraphics[width=0.48\textwidth,clip,trim=4cm 0cm 4cm 0cm]{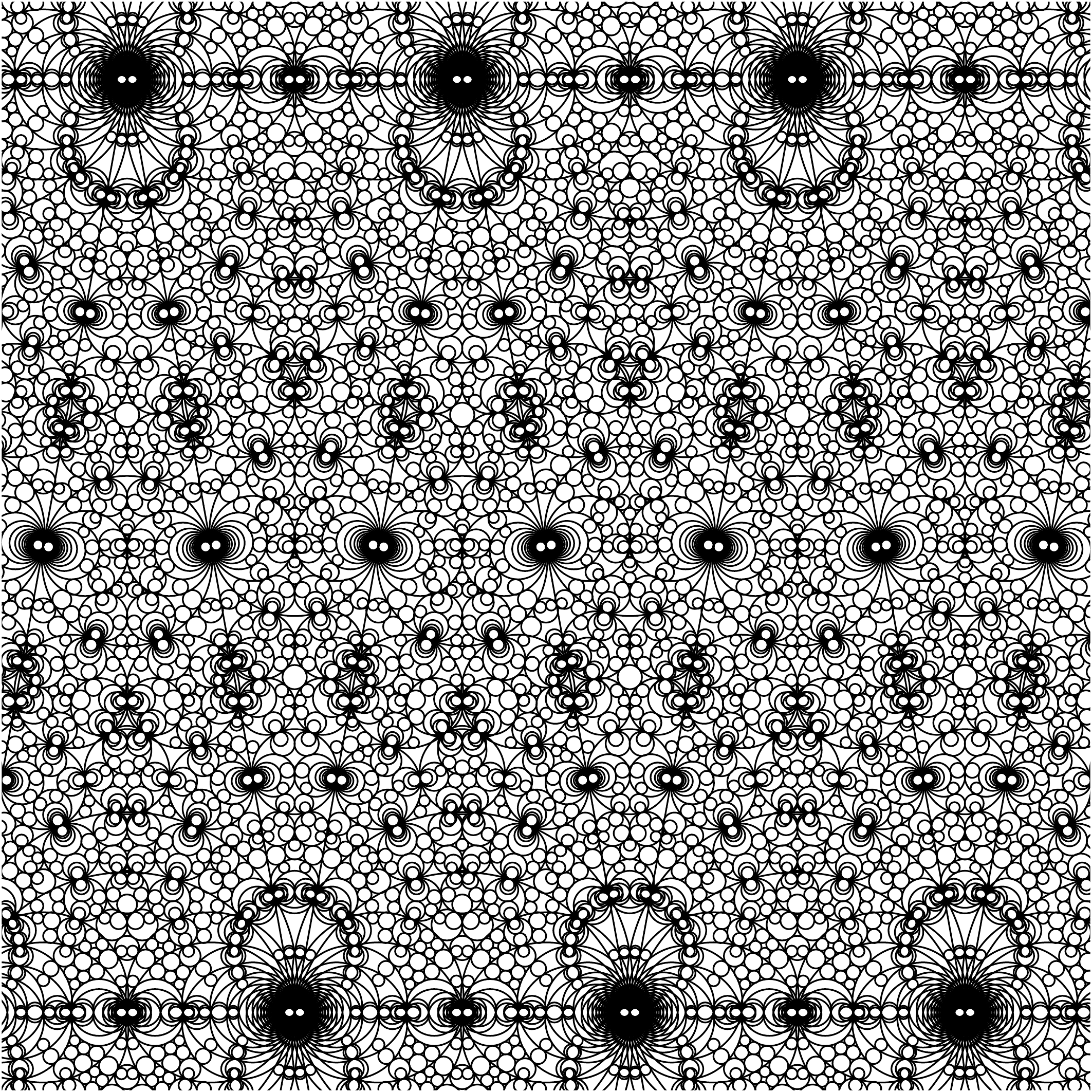}\hspace{0.04\textwidth}\includegraphics[width=0.48\textwidth,clip,trim=4cm 0cm 4cm 0cm]{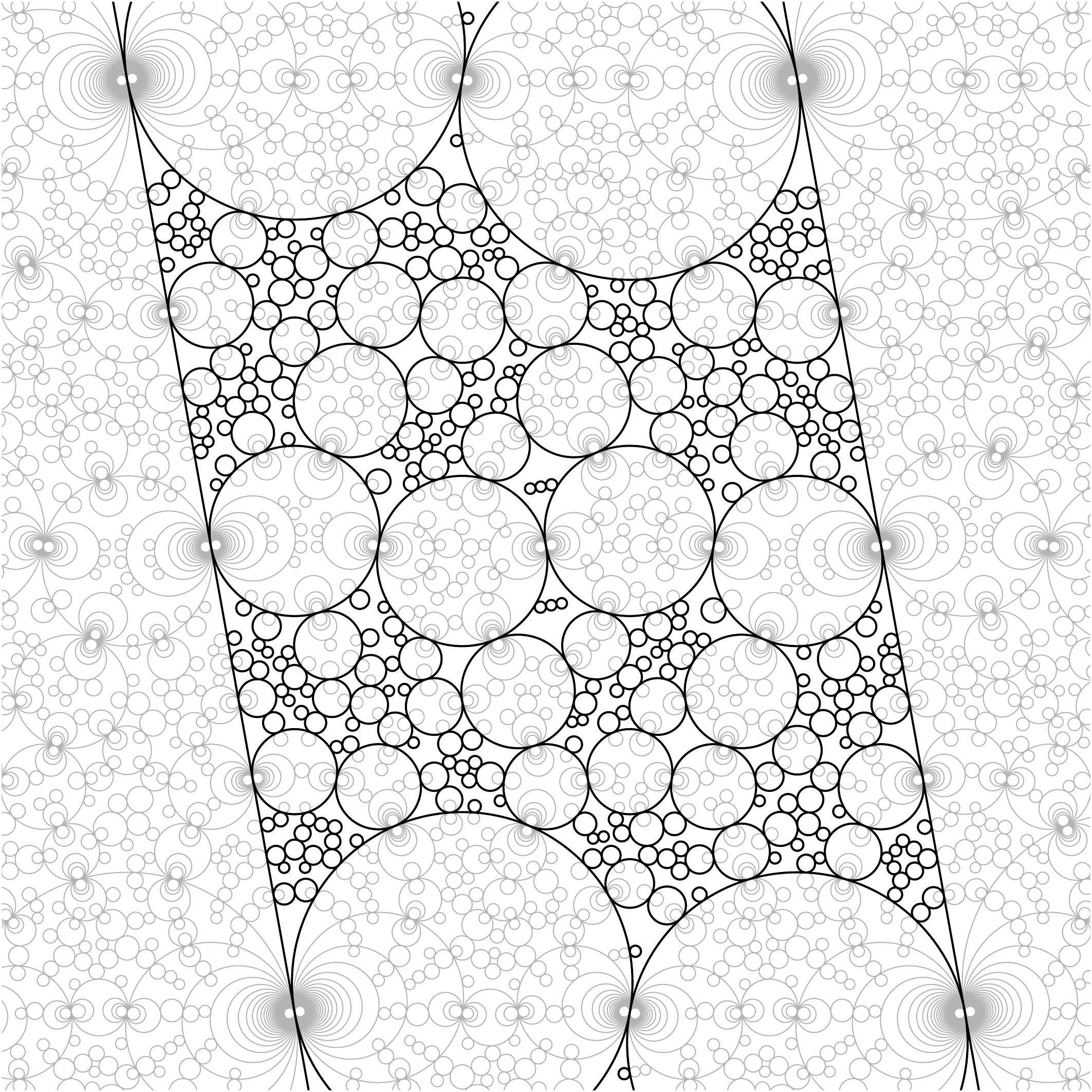}
    \captionsetup{width=1\textwidth}\caption{Piece of $\cS_8$ for $\bQ(i\sqrt{31})$ (left) and a possible choice of $\cA$ and immediate tangency packing $\cP$ (right) as per Definition \ref{def:itp}.}
    \label{fig:4}
\end{figure}

Recall that $H(D,\Delta)$ is the Hilbert symbol (Notation \ref{not:hilbert}).

\begin{definition}\label{def:itp}Suppose $H(D,\Delta)=1$. Let $\cA$ consist of those $C\in\cS_D$ for which $i\sqrt{D}\bv_C$ is in some fixed cosets of a full-rank sublattice of $\cO^2\times\sqrt{\Delta}\bZ^2$. An \emph{immediate tangency packing}, $\cP\subset\cA$, is a circle packing such that each $C\in\cA$ is contained in the closure of the interior of a single $C'\in\cP$.\end{definition}

While $\cO^2\times\sqrt{\Delta}\bZ^2$ has rank six over $\bR$, the first two entries of any $\bv_C$ are complex conjugates. So only a rank four sublattice of $\cO^2\times\sqrt{\Delta}\bZ^2$ is actually relevant to determining $\cA$ in Definition \ref{def:itp}.

Note that if $\cP\subset\cA$ is an immediate tangency packing and $C,C'\in\cP$ are tangent, then they must be ``immediately tangent" as defined by Stange \cite{stange2} and as described above regarding Figure \ref{fig:2}. That is, no circle of $\cA$ can be caught in between $C$ and $C'$, tangent and exterior to both. Such a circle could not possibly be in the closure of the interior of a single element of $\cP$. The immediate tangency property is not stated explicitly in Definition \ref{def:itp} due to ambiguity that arises when $\cA$ is not connected. Our definition uniquely defines a packing from an initial oriented circle whether $\cA$ is connected or not (as seen in the next proof).

Returning to Figure \ref{fig:4}, the sublattice defined by $i\sqrt{8}\zeta\in(2,(1-i\sqrt{31})/2)$ and $i\sqrt{8}\hat{r}\in (2)$ has four cosets because $(2,(1-i\sqrt{31})/2)$ has index two in $\cO$ and $(2)$ has index two in $\bZ$. Its trivial coset produces the background circles in the second image, which is Definition \ref{def:itp}'s $\cA$. Any of the three nontrivial cosets could also be used to create immediate tangency packings because $\cA$ in each case has only tangential intersections. It would be interesting to have a general method or criterion for selecting sublattices like this given some $\cS_D$. The one used for Figure \ref{fig:4} was found by experimentation.

\begin{corollary}\label{cor:local}All immediate tangency packings satisfy the almost local-global principle. \end{corollary}

\begin{proof}Let $\cP\subset \cA\subseteq\cS_D$ as in Definition \ref{def:itp}. Since the sublattice defining $\cA$ is assumed to have full rank, it has finite index in $\cO^2\times\sqrt{\Delta}\bZ^2$. Call the index $L$, and let $\Gamma(L)\leq\text{PSL}_2(\cO)$ be the congruence subgroup of level $L$. If $M\in\Gamma(L)$ then $N$ from (\ref{eq:2}) is congruent to the identity $\text{mod}\,L$, and thus fixes each sublattice coset. In particular, $\Gamma(L)$ fixes $\cA$. 

Use $C_{\text{in}}$ and $C_{\text{ex}}$ to denote the interior and exterior of some $C\in\cA$, and let $\overline{C_{\text{in}}}=C_{\text{in}}\cup C$ and $\overline{C_{\text{ex}}}=C_{\text{ex}}\cup C$ be their closures. Let $M\in\Gamma(L)$. We claim that $\cP\cap M\cP$ is either empty or $\cP$ as Theorem \ref{thm:local} requires. Let us suppose $C_0\in M\cP$ but $C_0\not\in\cP$ and aim to show that $\cP\cap M\cP=\emptyset$. By assumption, there is $C\in\cP$ with $C_0\subset\overline{C_{\text{in}}}$, as well as $C'\in \cP$ with $M^{-1}C\subset\overline{C'_{\text{in}}}$. Then $M\overline{C'_{\text{in}}}$ contains $C$ and thus contains either $C_{\text{in}}$ or $C_{\text{ex}}$. It must be the latter: Since $MC'$ and $C_0$ are both in the circle packing $M\cP$, $M\overline{C'_{\text{in}}}\supseteq \overline{C_{\text{in}}}\supset C_0$ forces $MC'$ and $C_0$ to be the same oriented circle. But $M'C=C=C_0$ contradicts $C\in\cP$ and $C_0\not\in\cP$. So as claimed, $M\overline{C'_{\text{in}}}$ contains $\overline{C_{\text{ex}}}$ and therefore all of $\cP$. The only element of $M\cP$ in $M\overline{C'_{\text{in}}}$ is $MC'$ itself, which is not in $\cP$ by the previous sentence. Thus $\cP\cap M\cP=\emptyset$.

To apply Theorem \ref{thm:local}, it remains to check that every circle in $\cP$ is tangent to at least one other. Fix an element of $\cP$ and write it as $M\widehat{\bR}$ for some $M\in\text{PGL}_2(K)$ with $|\det M|/\|(M)\|=\sqrt{D}$. Then $$M\widehat{\bR}\hspace{0.5cm}\text{and}\hspace{0.5cm}M\begin{bmatrix}1 & DL\sqrt{\Delta}\\ 0 & -1\end{bmatrix}\widehat{\bR}$$ are tangent with disjoint interiors since the same is true of $\widehat{\bR}$ oriented positively and $\widehat{\bR}-DL\sqrt{\Delta}$ oriented negatively. Call the second circle above $C$. Using the same argument from (\ref{eq:7}), we have $$C=\left(M\begin{bmatrix}1 & DL\sqrt{\Delta}\\ 0 & 1\end{bmatrix}M^{-1}\right)M\widehat{\bR}\in\Gamma(L) M\widehat{\bR}\subseteq\Gamma(L)\cA = \cA.$$ Thus $C\subset\overline{C'_{\text{in}}}$ for some $C'\in\cP$ by assumption. But then $C'$ and $M\widehat{\bR}$ must be tangent at $C\cap M\widehat{\bR}$ for their interiors to be disjoint.\end{proof}

\begin{corollary}All Baragar-Lautzenheiser circle packings satisfy the almost local-global principle.\end{corollary}

\begin{proof}The rank four sublattice of $\cO^2\times\sqrt{\Delta}\bZ^2$ from Definition \ref{def:itp} is generated by $\bv_{C_1}$, $\bv_{C_2}$, $\bv_{C_3}$, and $\bv_{C_4}$ from (\ref{eq:5}) and (\ref{eq:6}). Their circle packing is constructed to be an immediate tangency packing in the resulting arrangement $\cA$.\end{proof}

With respect to circle packings, the author suspects that Theorem \ref{thm:stange}'s hypothesis is not significantly weaker than Corollary \ref{cor:local}'s. It seems at least slightly weaker, since Fuchs, Stange, and Zhang do not require $M\Gamma N\widehat{\bR}$ to have tangential intersections, while immediate tangency packings always have tangential intersections. This difference aside, the author suspects that Definition \ref{def:itp} is essentially a geometric formulation of the algebraically-phrased hypothesis of Theorem \ref{thm:stange} in the special case of circle packings.

\section{Relation to the class group}\label{sec:5}

\subsection{Geometry at a point}\label{ss:point}We first study the relationship among oriented circles in $\cS_D$, if any, that contain a fixed point in $\widehat{K}$.

\begin{definition}A \emph{family} is a maximal subset of $\cS_D$ in which any two circles satisfy $\langle\bv_C,\bv_{C'}\rangle_Q=1$. An \textit{extended family} is a maximal subset in which any two circles intersect at a fixed point with angle $\theta$ satisfying $e^{i\theta}\in\cO$.\end{definition}

All circles in a family intersect tangentially at a fixed point in $\widehat{K}$ by Proposition \ref{prop:discrete}. Also remark that if $\Delta\neq-3,-4$ the only units in $\cO$ are $1$ and $-1$. So an extended family contains two families consisting of the same set of circles with opposite orientation. 

\begin{lemma}\label{lem:matform}If $C\in\cS_D$ and $\alpha\in C\cap K$, then $C=M\widehat{\bR}$ for some $M\in\emph{PGL}_2(K)$ with left column entries $\alpha$ (top) and $1$, $(M)=(\alpha,1)$, and $|\det M|/\|(M)\|=\sqrt{D}$.\end{lemma}

\begin{proof}That $C\cap K$ is nonempty implies $H(D,\Delta)=1$ by Lemma \ref{lem:split}, giving $M'\in\text{PGL}_2(K)$ with $M'\widehat{\bR}=C$ and $|\det M'|/\|(M')\|=\sqrt{D}$. We have $M'^{-1}(\alpha)\in \widehat{\bR}\cap\widehat{K}=\widehat{\bQ}$. So there is some $N\in\text{PSL}_2(\bZ)$ such that $N^{-1}M'^{-1}(\alpha)$ is the point at infinity, $1/0$. In particular, if $\alpha'$, $\beta'$, $\gamma'$ and $\delta'$ are the entries of $M'N$ (arranged as in (\ref{eq:2})), then $M'N(1/0)=\alpha$ means $\alpha'/\beta'=\alpha$. Consider the matrix $$M=M'N\begin{bmatrix}1/\beta' & 0\\ 0 & \|(\alpha,1)/(M')\|\overline{\beta'}\end{bmatrix} = \begin{bmatrix}\alpha & \|(\alpha,1)/(M')\|\overline{\beta'}\gamma' \\ 1 & \|(\alpha,1)/(M')\|\overline{\beta'}\delta'\end{bmatrix}\in\text{PGL}_2(K).$$ Recall that $(M'N)=(M')$ by Lemma \ref{lem:coprime}. The ideal generated by the top-right entry of $M$ is therefore $$\frac{\|(\alpha,1)\|(\overline{\beta'}\gamma')}{\|(M')\|}=(\alpha,1)\cdot\frac{(\gamma')}{(M')}\cdot\frac{(\overline{\alpha\beta'},\overline{\beta'})}{(\overline{M'})}=(\alpha,1)\cdot\frac{(\gamma')}{(M'N)}\cdot\frac{(\overline{\alpha'},\overline{\beta'})}{(\overline{M'N})},$$ which is contained in $(\alpha,1)$. The same is true of the bottom-right entry of $M$. This shows that $(M)=(\alpha,1)$, and $|\det M|/\|(M)\| = \|(\alpha,1)/(M')\||\det M'|/\|(\alpha,1)\| = |\det M'|/\|(M')\| = \sqrt{D}$.\end{proof}

We use $[\fa]$ to denote the ideal class of an ideal $\fa\subset K$.

\begin{theorem}\label{thm:angles}At $\alpha\in K$, there is one extended family per integral ideal of norm $D$ in $[(\alpha,1)]^2$. An intersection angle between elements in extended families corresponding to the ideals $\fD$ and $\fD'$ is the argument of a generator for $\fD/\fD'$. When scaled by $\sqrt{D/|\Delta|}$, curvatures from a single family form a congruence class $\emph{mod}\,1/\|(\alpha,1)\|$.\end{theorem}

\begin{proof}If $\fD_0\in [(\alpha,1)]^{2}$ is integral with norm $D$, then we can find $\gamma,\delta\in (\alpha,1)\fD_0$ with $(\alpha\delta-1\gamma)=(\alpha,1)^2\fD_0$. So let $M_0$ have entries as in (\ref{eq:2}) but with $\beta=1$. Then $(M_0)=(\alpha,1)$ and $|\det M_0|/\|(M_0)\|=\sqrt{\|\fD_0\|}=\sqrt{D}$. 

For an integral ideal $\fD\in[(\alpha,1)]^2$ of norm $D$, we will show that \begin{equation}\label{eq:8}\left\{M_0\begin{bmatrix}1 & \lambda \\ 0 & \mu\end{bmatrix}\widehat{\bR}\;\bigg|\;\lambda\in\cO,\,(\mu)=\fD/\fD_0\right\}\end{equation} is an extended family, and that a family is defined by fixing the generator $\mu$. 

But first, the angle between two circles from (\ref{eq:8}), one with matrix entries $\lambda$ and $\mu$ as above and the other with entries $\lambda'$ and $\mu'$, is the angle between $\widehat{\bR}$ and \begin{equation}\left(M_0\begin{bmatrix}1 & \lambda\\ 0 & \mu\end{bmatrix}\right)^{\!-1}\!M_0\begin{bmatrix}1 & \lambda'\\ 0 & \mu'\end{bmatrix}\widehat{\bR}=\begin{bmatrix}1 & (\lambda'\mu-\lambda\mu')/\mu\\ 0 & \mu'/\mu\end{bmatrix}\widehat{\bR}.\label{eq:9}\end{equation} This angle is evidently the argument of $\mu'/\mu$, which generates $\fD'/\fD$ assuming $(\mu)=\fD/\fD_0$ and $(\mu')=\fD'/\fD_0$. In particular, if $\fD=\fD'$ then $\mu'/\mu\in\cO$, showing that each set of the form (\ref{eq:8}) is contained in a single extended family. Moreover, if $\mu=\mu'$ then the oriented circle in (\ref{eq:9}), call it $C$, has curvature-center $i$ by Lemma \ref{lem:stange}. Since $\bv_{\widehat{\bR}}=[i\;-i\;\,0\;\,0]$, we get $\langle\bv_{\widehat{\bR}},\bv_C\rangle_Q=1$. So for fixed $\mu$, (\ref{eq:8}) is contained in a single family.  

Next, using the formula from Lemma \ref{lem:stange}, scaling a curvature from (\ref{eq:8}) by $\sqrt{D/|\Delta|} = |\mu\det M_0|/\sqrt{|\Delta|}\|(\alpha,1)\|$ gives $$\frac{2\Im(\mu\delta)}{\sqrt{\Delta|}\|(\alpha,1)\|} + \frac{2\Im(\lambda)}{\sqrt{\Delta|}\|(\alpha,1)\|}.$$ By fixing $\mu$ and varying $\lambda$ we obtain curvatures from the same family that produce a full congruence class $\text{mod}\,1/\|(\alpha,1)\|$.

It remains only to check that every circle in $\cS_D$ containing $\alpha$ can be expressed as in (\ref{eq:8}) for some $\lambda$ and $\mu$. By Lemma \ref{lem:matform}, every such circle is of the form $M\widehat{\bR}$ for some $M\in\text{PGL}_2(K)$ with left column entries $\alpha$ and $1$, $(M)=(\alpha,1)$, and $(\det M)/(M)^2=\fD$ for some $\fD\subseteq\cO$ of norm $D$. Let $\gamma'$ and $\delta'$ denote the right column entries of $M$. We have $$M_0^{-1}M=\begin{bmatrix}1 & (\gamma'\delta-\gamma\delta')/\det M_0 \\ 0 & \det M/\det M_0\end{bmatrix}.$$ First observe that $(\det M/\det M_0)=\fD(M)^2/\fD_0(M_0)^2=\fD(\alpha,1)^2/\fD_0(\alpha,1)^2=\fD/\fD_0$. Next observe that $(\gamma'\delta-\gamma\delta')\in (\gamma',\delta')(\gamma,\delta)\subseteq (M)(\gamma,\delta)=(\alpha,1)(\gamma,\delta)=(\alpha,1)^2\fD_0=(\det M_0)$. Therefore the upper-right entry above is some integer $\lambda$.\end{proof}

As an example, recall from Figure \ref{fig:3} the claim that certain ideal classes of $\bQ(i\sqrt{39})$ can be distinguished in $\cS_4$. This field has a cyclic class group of order four, generated by $[\fp_2]$ for either prime $\fp_2$ over $2$. The only integral ideal of norm $4$ in the principal ideal class is $(2)$. So according to the theorem, if $[(\alpha,1)]^2=[\cO]$ (meaning $(\alpha,1)\in[\cO]$ or $(\alpha,1)\in[\fp_2]^2$) then $\cS_4$ has only one extended family at $\alpha$. These are points in Figure \ref{fig:3} with only tangential intersection, as $\bZ[(1+i\sqrt{39})/2]$ has trivial unit group. On the other hand, there are two integral ideals of norm $4$ in $[\fp_2]^2$: $\fp_2^2$ and $\overline{\fp_2}^2$. Therefore when $(\alpha,1)\in[\fp_2]$ or $(\alpha,1)\in[\fp_2]^3$ there are two extended families at $\alpha$. The angle between them is the argument of a generator for $\fp_2^2/\overline{\fp_2}^2$, which is $(5\pm i\sqrt{39})/8$ depending on the choice of $\fp_2$. 

\begin{definition}\label{def:cover}An arrangement $\cA$ \textit{covers} $\alpha$ if $\alpha\in C$ for some $C\in \cA$. In this case, if $\alpha\in K$ we say $\cA$ \textit{covers} the ideal class corresponding to $\alpha$, which is $[(\alpha,1)]$.\end{definition}

\begin{corollary}\label{cor:classes}A point $\alpha\in K$ is covered by $\cS_D$ if and only if $[(\alpha,1)]^2$ contains an integral ideal of norm $D$. In particular, whether or not $\alpha$ is covered depends only the coset of the $2$-torsion subgroup of the class group to which $[(\alpha,1)]$ belongs.\end{corollary}

\begin{proof}A point $\alpha\in \widehat{K}$ is covered if and only if $\cS_D$ has at least one extended family at $\alpha$. By Theorem \ref{thm:angles}, this is equivalent to $[(\alpha,1)]^2$ containing at least one integral ideal of norm $D$.\end{proof}

So in $\cS_D$, covering an ideal class implies every point corresponding to that ideal class is covered, and indeed every point corresponding to an ideal class in the same $2$-torsion coset.

We can also say something about the geometry of $\cS_D$ around $\alpha\in K$ that are not covered. These points exhibit the typical repulsion property of rational numbers in Diophantine approximation. Figure \ref{fig:5} shows $\cS_{39}$ for $\bQ(i\sqrt{39})$, which only covers the two $2$-torsion ideal classes in the class group. A red forbidden zone has been drawn around uncovered points in $K$ like $(1+i\sqrt{39})/4$, the focal point of the image. The radius of a red dot is computed from the next proposition with $r=1200$, because Figure \ref{fig:5} shows curvatures up to $1200$.

\begin{proposition}\label{prop:holes}If $\alpha\in K$ is not on some $C\in\cS_D$ of curvature $r$, the distance between $C$ and $\alpha$ is at least $(\sqrt{2}-1)\min(d,\sqrt{d/|r|})$, where $d=\sqrt{|\Delta|}\|(\alpha,1)\|/\sqrt{D}$.\end{proposition}

\begin{proof}Suppose $\alpha\not\in C\in\cS_D$. Let $\bv_C=[\zeta\;\,\overline{\zeta}\;\,\hat{r}\;\,r]$. Recall that when $r=0$, $\zeta$ is a unit vector orthogonal to $C$. So the distance between $\alpha$ and $C$ is the scalar projection of $\alpha$ onto $\zeta$, which is $|\Re(\zeta\overline{\alpha})|=|2\Im(i\sqrt{D}\zeta\overline{\alpha})|/2\sqrt{D}$. As $2\Im(i\sqrt{D}\zeta\overline{\alpha})$ is an integer multiple of $\sqrt{|\Delta|}\|(\alpha,1)\|$, it follows that $|\Re(\zeta\overline{\alpha})|\geq d/2 > (\sqrt{2}-1)d$.

Now suppose $r\neq 0$. From $|\zeta|^2=1+\hat{r}r$ we have the second equality below: $$D|\zeta-r\alpha|^2=D(|\zeta|^2-2r\Re(\zeta\overline{\alpha})+r^2|\alpha|^2)=$$ $$D-i\sqrt{D}r\big(i\sqrt{D}\hat{r}-2i\Im(i\sqrt{D}\zeta\overline{\alpha})+i\sqrt{D}r|\alpha|^2\big).$$ In parentheses on the bottom line is a nonzero (since $\alpha\not\in C$) integer multiple of $\sqrt{\Delta}\|(\alpha,1)\|$. In particular, $D|\zeta-r\alpha|^2$ is either at least $D+\sqrt{D|\Delta|}\|(\alpha,1)\||r|=D(1+d|r|)$ or at most $D-\sqrt{D|\Delta|}\|(\alpha,1)\||r|=D(1-d|r|)$. Thus if $|r|<1/d$, the distance between $C$ and $\alpha$ is at least $$\left|\left|\frac{\zeta}{r}-\alpha\right|-\frac{1}{|r|}\right|\geq \frac{1}{|r|}\min\left(\sqrt{1+d|r|}-1,-\sqrt{1-d|r|}+1\right)>\frac{\sqrt{2}-1}{|r|} >(\sqrt{2}-1)d.$$ 

On the other hand, if $|r|\geq 1/d$ then it is not possible for $D|\zeta-r\alpha|^2$, which is positive, to be at most $D(1-d|r|)$. So there is no need for the min function above. The distance between $C$ and $\alpha$ is at least $$\frac{1}{|r|}\left(\sqrt{1+d|r|}-1\right)=\sqrt{\frac{d}{|r|}}\left(\sqrt{\frac{1}{d|r|}+1}-\sqrt{\frac{1}{d|r|}}\right)\geq(\sqrt{2}-1)\sqrt{\frac{d}{|r|}},$$ where the last inequality uses $|r|\geq 1/d$ again.\end{proof}

\begin{figure}[t]
    \centering
    \includegraphics[height=0.64\textwidth,clip,trim=0cm 0cm 0cm 6.8cm]{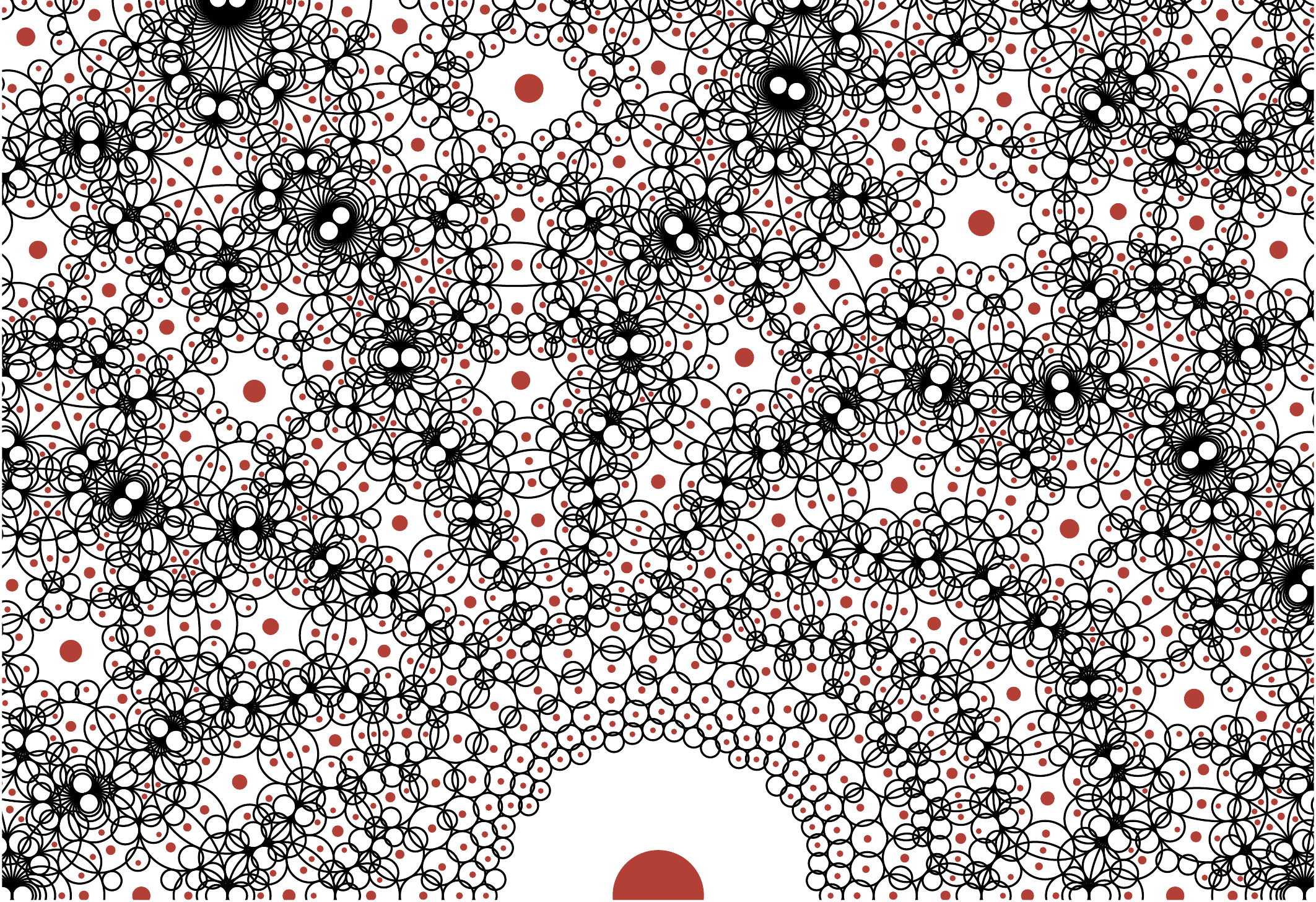}
    \captionsetup{width=0.936\textwidth}
    \caption{$\cS_{39}$ for $\bQ(i\sqrt{39})$ keeping away from uncovered rational points.}\label{fig:5}
\end{figure}

\subsection{Connectivity and the class group} Stange proved that $\cO$ is Euclidean if and only if the Schmidt arrangement (or $\cS_1$) is topologically connected \cite{stange}. The relationship with Euclideaneity is extended to arbitrary $\cS_D$ in the author's dissertation, where it is shown how each step in a pseudo-Euclidean algorithm corresponds to one ``step" along a chain of circles \cite{martin}. This observation inspired a continued fraction algorithm in which the resulting approximations are chain intersection points from the ``walk" through $\cS_D$. The forthcoming example surrounding Figure \ref{fig:6} hints at the algorithm, but it is not stated here (see \cite{martin2}). The focus of this subsection is consequences of such pseudo-Euclideaneity for the class group. This is to say that Theorem \ref{thm:fincon} and Corollary \ref{cor:finrat} should be thought of as generalizations of the classical statement ``Euclidean implies principal ideal domain."

\begin{lemma}\label{lem:classes}Let $M\in\emph{PGL}_2(K)$. If $\alpha\in M\widehat{\bR}\cap K$ then $[(\alpha,1)]=[(M)\fd]$ for some integral ideal $\fd$ dividing $(\det M)/(M)^2$.\end{lemma}

\begin{proof}Let $M$ have entries as in (\ref{eq:2}). Any point in $C\cap K$ is of the form $M(p/q) = (\alpha p+\gamma q)/(\beta p+\delta q)$ for some $p/q\in\widehat{\bQ}$. Assuming $p,q\in\bZ$ are coprime, we have $(\det M)/(M)\subseteq (\alpha p+\gamma q,\beta p+\delta q)\subseteq (M)$ by Lemma \ref{lem:coprime}.  Then setting $\fd = (\alpha p+\gamma q,\beta p+\delta q)/(M)$ gives $(\det M)/(M)^2\subseteq\fd\subseteq\cO$ as well as $[(M(p/q),1)]=[(\alpha p+\gamma q,\beta p+\delta q)] = [(M)\fd]$.\end{proof}

Note that the quotient of ideal classes corresponding to two points on the same circle in $\cS_D$ is $[\fd/\fd']$ for some $\fd,\fd'\subseteq\cO$ with norms dividing $D$. This prompts the next definition.  

\begin{definition}\label{def:ratcon}An arrangement $\cA$ is \emph{rationally connected} if for any $C,C'\in\cA$ there is a chain $C=C_0,C_1,...,C_n=C'\in \cA$ such that $C_{k-1}\cap C_k\cap\widehat{K}\neq\emptyset$ for all $k$.\end{definition}

\begin{proposition}\label{prop:ratcon}The ideal classes covered by a rationally connected subset of $\cS_D$ are contained in a single coset of the subgroup generated by ideal classes of primes containing $D$.\end{proposition}

\begin{proof}If $H(D,\Delta)=-1$ then $\cS_D$ covers no points of $\widehat{K}$ by Lemma \ref{lem:split}. The claim holds vacuously in this case. 

If $H(D,\Delta)=1$, then Theorem \ref{thm:redef} says every circle in $\cS_D$ is of the form $M\widehat{\bR}$, where $M\in\text{PGL}_2(K)$ and $(\det M)/(M)^2$ has norm $D$. By Lemma \ref{lem:classes}, the ideal classes covered by such a circle are contained in a single coset of the subgroup generated by ideal classes of primes containing $D$. So our claim follows by induction on $n$ from Definition \ref{def:ratcon}.\end{proof}

\begin{proposition}\label{prop:DDelta}The ideal classes covered by $\cS_D$ are contained in a single coset of the subgroup generated by ideal classes of primes containing $D\Delta$.\end{proposition}

\begin{proof}

As in the previous proof, we are done unless $H(D,\Delta)=1$. So fix $C_0\in\cS_D$ and use Theorem \ref{thm:redef} to write $C_0=M_0\widehat{\bR}$ as usual. Let $\Gamma_{\Delta}$ and $\Gamma_{D\Delta}$ denote the subgroups generated by ideal classes of primes containing $\Delta$ and $D\Delta$. We claim every ideal class covered by $\cS_D$ is in $[(M_0)]\Gamma_{D\Delta}$. 

Let $C\in\cS_D$ be arbitrary and write $C=M\widehat{\bR}$. All ideal classes covered by $C$ are in $[(M)]\Gamma_D$ by Lemma \ref{lem:classes}, so we will be done if we can show $[(M)]\in [(M_0)]\Gamma_{D\Delta}$. Let $\fD_0=(\det M_0)/(M_0)^2$ and $\fD=(\det M)/(M)^2$. There are at most two ideals in $\cO$ above each prime in $\bZ$---a prime ideal and its conjugate---so $\|\fD_0\|=\|\fD\|$ implies $\fD_0$ and $\fD$ have the same set of prime ideal divisors up to conjugation. In particular, there is some integral ideal $\fa\,|\,\fD_0$ satisfying $\overline{\fa}\fD_0=\fa\fD$. Now divide both sides of $[(\det M)]=[(\det M_0)]$ by $[\fD]$ to get $[(M)]^2=[(M_0)^2\fD_0/\fD]=[(M_0)^2\fa/\overline{\fa}]=[(M_0)\fa]^2$. Thus $[(M)]$ and $[(M_0)\fa]$ are in the same coset of the 2-torsion subgroup, which is exactly $\Gamma_{\Delta}$. Finally, $\|\fa\|\,|\,D$ gives $[(M)]\in[(M_0\fa)]\Gamma_{\Delta}\subseteq[(M_0)]\Gamma_{D\Delta}$.\end{proof}

We want to say something substantive about the size of the cosets in the two propositions above, but rational connectivity is not enough by itself. Let us show why with an example.

Recall from Figure \ref{fig:5} that $\cS_{39}$ for $\bQ(i\sqrt{39})$ only covers one coset (out of two in this case) of the 2-torsion subgroup of the class group---the bare minimum for any $\cS_D$ when $H(D,\Delta)=1$ by Corollary \ref{cor:classes}. Nevertheless, it appears that every circle is tangent to a larger one, and such intersections are in $\widehat{K}$ by Proposition \ref{prop:discrete}. 

To prove that $\cS_{39}$ is indeed rationally connected, we will show that each $C_0$ is tangent to some $C_1$ such that $\max\|(\alpha,1)\|$ among $\alpha\in C_1\cap K$ is strictly larger than among $\alpha\in C_0\cap K$ (provided the latter is less than $1$). This creates a chain $C_0,C_1,...,C_n$, where $C_n$ contains some $\alpha$ with $\|(\alpha,1)\|=1$ (meaning $\alpha\in\cO$). By Theorem \ref{thm:angles}, a family at such an $\alpha$ contains a circle of curvature $0$. All circles of curvature $0$ meet at the point at infinity, thereby proving rational connectivity. 

Given some $C_0\in\cS_{39}$ with nonzero curvature, let $\alpha\in C_0\cap K$ be such that $\|(\alpha,1)\|$ is maximal among such points. Using Lemma \ref{lem:matform} we may write $C_0=M\widehat{\bR}$ for some $M\in\text{PGL}_2(\bQ(i\sqrt{39}))$ with left-column entries $\alpha$ and $1$, $(M)=(\alpha,1)$, and $|\det M|/\|(M)\|=\sqrt{39}$. By right multiplying $M$ by the appropriate upper-triangular matrix in $\text{PSL}_2(\bZ)$, which does not affect $(M)=(\alpha,1)$ by Lemma \ref{lem:coprime}, we may further assume that its right-column entries, say $\gamma$ and $\delta$, are in $(M)(i\sqrt{39})$.

For each $\lambda\in\cO$ there is a circle tangent to $C_0$: $$C_{\lambda}=M\begin{bmatrix}1 & \lambda\\ 0 & 1\end{bmatrix}\widehat{\bR}.$$ Note that $M(\lambda/1)=(\alpha\lambda+\gamma 1)/(1\lambda+\delta 1)\in C_{\lambda}$. Also, by Lemma \ref{lem:coprime}, $(\lambda,1)=\cO$ implies $(\alpha\lambda+\gamma,\lambda+\delta)\supseteq(\det M)/(M) = (\alpha,1)(i\sqrt{39})$. But $\gamma,\delta\in(\alpha,1)(i\sqrt{39})$, which justifies the second equality below: \begin{equation}\label{eq:10}\left\|\left(\frac{\alpha\lambda+\gamma}{\lambda+\delta},1\right)\right\|=\frac{\|(\alpha\lambda+\gamma,\lambda+\delta)\|}{|\lambda+\delta|^2}=\frac{\|(\alpha,1)(\lambda,i\sqrt{39})\|}{|\lambda+\delta|^2}.\end{equation} The norm above exceeds $\|(\alpha,1)\|$ if and only if $|\lambda+\delta|^2<\|(\lambda,i\sqrt{39})\|$. That is, we win with $C_1=C_{\lambda}$ if $-\delta$ is contained in the open disc centered on $\lambda\in\cO$ with radius squared $\|(\lambda,i\sqrt{39})\|$. Such discs cover the plane!..almost.

Our open discs leave holes (akin to ``singular points" on the Ford domain for $\text{PSL}_2(\cO)$ \cite{swan}), three of which can be seen along the vertical centerline of Figure \ref{fig:6}. This image shows a fundamental region for $(i\sqrt{39})$. The holes occur at the two cosets of $(i\sqrt{39})$ represented by $(39\pm i\sqrt{39})/4$. We claim $-\delta$ cannot be such a point, which would complete the argument that $\cS_{39}$ is rationally connected. Note that $((39\pm i\sqrt{39})/2,2)$ is a prime over $2$, which is not in a 2-torsion ideal class. On the other hand, $\delta$ lies on the image of $\widehat{\bR}$ under the transpose of $M$. This is a circle in $\cS_D$, implying $(\delta,1)=(-\delta,1)$ is in a $2$-torsion ideal class.

\begin{figure}[t]
    \centering
    \includegraphics[width=0.79\textwidth,clip,trim=0cm 5.4cm 0cm 5.5cm]{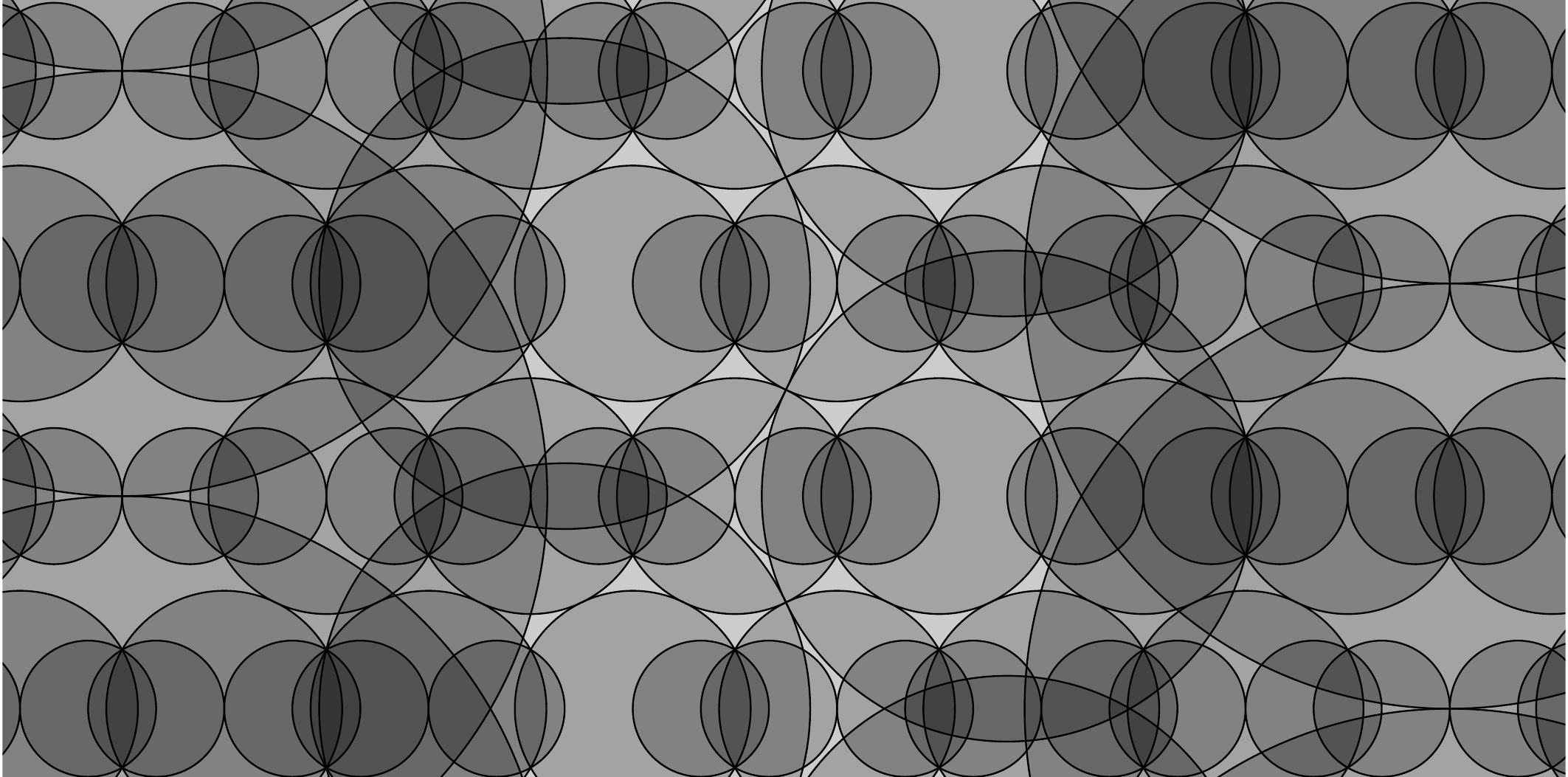}
    \captionsetup{width=0.79\textwidth}
    \caption{Discs centered on $\lambda\in\cO$ of radius squared $\|(\lambda,i\sqrt{39})\|$.}\label{fig:6}
\end{figure}

Although $-\delta$ cannot actually land on a hole, it can come arbitrarily close. Thus the ratio of $\|(\alpha,1)\|$ to the norm in (\ref{eq:10}) can be arbitrarily close to $1$, making for a small gain with our step from $C_0$ to $C_1=C_{\lambda}$. The problem is compounded by the fact that $C_1$ will experience the same issue. The value of ``$-\delta$" will slowly inch away from the nearby hole as we create the chain $C_0,C_1,...$. So while $\cS_{39}$ may be rationally connected, there is no upper bound on the chain length required to make significant progress toward reaching a circle of curvature $0$. Such is the insufficiency of rational connectivity, remedied below.

\begin{definition}\label{def:fincon}An arrangement $\cA$ is \emph{finitely connected} if there is $n\in\bN$ such that any $C\in\cA$ with nonzero curvature $r$ has a chain $C=C_0,C_1,...,C_n\in\cA$ for which $C_{k-1}\cap C_k\neq\emptyset$ for all $k$ and $C_n$ has curvature magnitude at most $|r|/2$.\end{definition}

The $C_k$'s need not be distinct, meaning the chain can have length less than $n$. Also remark that we do not require $C_{k-1}\cap C_k\subset\widehat{K}$.

The proof that $\cS_{4}$ for $\bQ(i\sqrt{39})$ in Figure \ref{fig:3} is finitely (and rationally) connected proceeds by finding disc covers of $\bC$ as in Figure \ref{fig:6}. In this case, however, three disc covers must be computed since there are three integral ideals of norm $4$. (For the previous example, there is only one integral ideal of norm $39$.) A nearly identical scenario is worked out in detail ($\cS_4$ for $\bQ(i\sqrt{31})$) in the author's dissertation \cite{martin}.

\begin{theorem}\label{thm:fincon}If $\cS_D$ is finitely connected, then it covers all of $\widehat{K}$. In particular, the class group is generated by ideal classes of primes with norm dividing $D\Delta$, and every ideal class in the principal genus contains an integral ideal of norm $D$.\end{theorem}

\begin{proof}Suppose $\cS_D$ is finitely connected with chain lengths $n$ as in Definition \ref{def:fincon}. We will show that an arbitrary $\alpha\in K$ is covered by $\cS_D$. Fix any $C\in\cS_D$ and $\tau\in C$ (in $K$ or not) which is not the point at infinity. Let $\bv_{C}=[\zeta\;\,\overline{\zeta}\;\,r\;\,\hat{r}]$.

Since $1\in (\alpha,1)$, there are infinitely many choices of coprime $\alpha_0,\beta_0\in\cO$ satisfying $\beta_0\alpha-\alpha_01=1$. Given such a pair, let $M_0\in\text{PSL}_2(\cO)$ be a matrix with left column entries $\alpha_0$ and $\beta_0$. Call its right column entries $\gamma_0$ and $\delta_0$, and assume $\delta_0$ has minimal magnitude in its coset $\text{mod}\,\beta_0$ (achieved via right multiplication by the appropriate upper triangular matrix). By Proposition \ref{prop:exiso}, the curvature of $M_0C$ is $2\Re(\zeta\beta_0\overline{\delta_0})+|\beta_0|^2\hat{r}+|\delta_0|^2r$. Our choice to make $|\delta_0|$ minimal bounds this curvature in magnitude from above by $c|\beta_0|^2$ for some constant $c$ independent of $\beta_0$. Also, the distance between $M_0C$ and $\alpha$ is at most $|M_0(\tau)-\alpha|=$ $$\left|\frac{\alpha_0\tau + \gamma_0}{\beta_0\tau + \delta_0}-\alpha\right|\leq\left|\frac{\alpha_0\tau + \gamma_0}{\beta_0\tau + \delta_0}-\frac{\alpha_0}{\beta_0}\right|+\left|\frac{\alpha_0}{\beta_0}-\alpha\right|=\frac{1}{|\beta_0(\beta_0\tau+\delta_0)|}+\frac{1}{|\beta_0|},$$ which is bounded above by $c'/|\beta_0|$ for some constant $c'$ independent of $\beta_0$.

Set $d=\sqrt{|\Delta|}\|(\alpha,1)\|/\sqrt{D}$, just as in Proposition \ref{prop:holes}. First fix $m\in\bN$ large enough so that \begin{equation}\label{eq:11}c'< \frac{(\sqrt{2}-1)\sqrt{2^md}}{2\sqrt{c}}.\end{equation} Then fix $M_0$ as above with $\beta_0$ large enough so that \begin{equation}\label{eq:12}\frac{2^{m+1}mn}{c|\beta_0|}< \frac{(\sqrt{2}-1)\sqrt{2^md}}{2\sqrt{c}}\hspace{0.5cm}\text{and}\hspace{0.5cm}\frac{2^m}{c|\beta_0|^2}< d.\end{equation} Set $C_0=M_0C$ and let $r_0$ be its curvature. Take a minimal-length chain of distinct circles $C_0,C_1,...,C_j\in\cS_D$ such that $|r_j|$, the curvature magnitude of $C_j$, is at most $c|\beta_0|^2/2^m$. Since $c$ is such that $|r_0|\leq c|\beta_0|^2$, we have $j\leq mn$ by finite connectivity. The diameter of $C_k$ if $0\leq k<j$ (if there is such an index) is less than $2(2^m/c|\beta_0|^2)$ by minimality of $j$. So the distance between $C_j$ and $\alpha$ is at most $$|M_0(\tau)-\alpha|+\frac{2^{m+1}j}{c|\beta_0|^2}\leq\frac{c'}{|\beta_0|}+\frac{2^{m+1}mn}{c|\beta_0|^2}<$$ $$(\sqrt{2}-1)\sqrt{\frac{2^md}{c|\beta_0|^2}}\leq (\sqrt{2}-1)\min\!\left(d,\sqrt{\frac{d}{|r_j|}}\right).$$ The middle inequality above uses (\ref{eq:11}) to bound its first summand and the first half of (\ref{eq:12}) to bound its second summand. The last inequality above uses $|r_j|\leq c|\beta_0|^2/2^m$ and the second half of (\ref{eq:12}). Thus $\alpha\in C_j$ by Proposition \ref{prop:holes}.

The final two claims of the theorem are Proposition \ref{prop:DDelta} and Corollary \ref{cor:classes}.\end{proof}

Besides Theorem \ref{thm:fincon}, our number theory results had to assume that $\cS_D$ has intersections in $\widehat{K}$. It would be interesting to find a criterion that guarantees finite connectivity and takes advantage of Definition \ref{def:fincon} not requiring $C_{i-1}\cap C_i\subset\widehat{K}$. The only way to prove finite connectivity of which the author is currently aware is to verify the covering property that $\cS_{39}$ for $\bQ(i\sqrt{39})$ narrowly failed in Figure \ref{fig:6}. But this property may be stronger than necessary as it also guarantees rational connectivity, which strengthens the conclusion of Theorem \ref{thm:fincon} as follows. 

\begin{corollary}\label{cor:finrat}If $\cS_D$ is finitely and rationally connected, then the class group is generated by ideal classes of primes with norm dividing $D$.\end{corollary}

\begin{proof}We are combining Proposition \ref{prop:ratcon} with Theorem \ref{thm:fincon}.\end{proof}

\bibliographystyle{plain}
\bibliography{refs}

\end{document}